\documentclass[hidelinks,onefignum,onetabnum]{siamart251216}

\usepackage{lipsum}
\usepackage{amsfonts}
\usepackage{graphicx}
\usepackage{epstopdf}
\usepackage{algorithmic}
\usepackage{ntheorem}
\ifpdf
  \DeclareGraphicsExtensions{.eps,.pdf,.png,.jpg}
\else
  \DeclareGraphicsExtensions{.eps}
\fi

\usepackage{comment}
\usepackage{mathrsfs}
\usepackage{enumitem}
\makeatletter
\def\my@fnsymbol#1{\ensuremath{\ifcase#1\or *\or  \mathsection \or \ddagger \or \dagger \or \mathparagraph\or \|\or \#\or \dagger\dagger
  \or \ddagger\ddagger\else\@ctrerr\fi}}
\renewcommand*{\fnsymbol}[1]{\expandafter\my@fnsymbol\csname c@#1\endcsname}
\makeatother

\newsiamremark{remark}{Remark}
\newsiamremark{observation}{Observation}
\newsiamremark{hypothesis}{Hypothesis}
\newsiamremark{notation}{Notation}
\crefname{hypothesis}{Hypothesis}{Hypotheses}
\newsiamthm{claim}{Claim}
\newsiamremark{fact}{Fact}
\crefname{fact}{Fact}{Facts}
\newsiamremark{problem}{Problem}

\headers{Mean-Field Oscillator Ising Machines}{A. R. Venkatakrishnan, M. Emerick, B. Bamieh, and F. Bullo}

\title{Mean-Field Oscillator Ising Machines: Gradient Flows and Classification of Limit Solutions\thanks{Submitted to the editors October 7, 2026.
\funding{Supported in part by ARO MURI W911NF-24-1-0228, NSF ECCS-2453491, and AFOSR FA9550-23-F-0014.}}}

\author{Arvind R. Venkatakrishnan\thanks{Department of Mechanical Engineering, University of California, Santa Barbara, CA, USA (\{bamieh,bullo,memerick,arvindragghav\}@ucsb.edu). The first two authors contributed equally.}
\and Max Emerick\footnotemark[2]
\and Bassam Bamieh\footnotemark[2]
\and Francesco Bullo\footnotemark[2]}

\usepackage{amsopn}

\newcommand{\leb}{\mathscr{L}}

\ifpdf
\hypersetup{
  pdftitle={title},
  pdfauthor={author}
}
\fi

\begin{document}

\maketitle

\begin{abstract}
Oscillator Ising Machines (OIMs) have emerged as promising computational architectures for approximating solutions to combinatorial optimization problems. We derive and analyze the mean-field limit of an OIM model and show that it inherits the gradient flow structure of the finite-dimensional dynamics. We identify conditions under which this mean-field evolution admits an Eulerian formulation as a gradient flow on the Wasserstein space of probability measures and contrast this with a Lagrangian formulation which is always available. This gradient flow structure strongly constrains the long-time dynamics and enables a complete classification of limit solutions and their stability in the totally symmetric case. In particular, all limit solutions are fixed points whose phases cluster into at most four groups, and for almost all parameter values, only binarized fixed points -- those with clusters at $0$ and/or $\pi$ -- can be stable. Since binarized states are exactly those for which a feasible solution to the original problem can be read out, this shows that feasible solutions can almost always be recovered. We provide tight bounds on the parameter thresholds for which fixed points in this binarized family are stable, thereby identifying the threshold for binarization in this model, and present numerical evidence that the mean-field model correctly predicts behavioral regimes in certain large random networks, including Erd\H{o}s-R\'enyi networks. Moreover, this problem provides a case study in the application of mean-field models, gradient flows, and optimal transport to the study of large-scale systems, and we highlight a number of features of this problem which are informative in this respect.
\end{abstract}



\section{Introduction}

Combinatorial Optimization Problems (COPs) are ubiquitous in modern networked systems, with applications including routing, scheduling, resource allocation, and data compression. Many of these problems are NP-complete, with the cost of obtaining exact solutions scaling exponentially with the problem size. This has driven the development of alternative computational architectures, such as Oscillator Ising Machines (OIMs), which can approximate solutions to these problems at a fraction of the computational cost.

As shown in \cite{AL::2014::FrontPhys}, various NP-complete COPs (and thereby all 21 of Karp's equivalent NP-complete problems \cite{RK::2010::IntProg}) can be reformulated as discrete minimization problems of the following form
\begin{equation} \label{ising_problem}
	\min_{\sigma \in \Sigma} ~ \frac{1}{2}\sum_{i,j} A_{ij} \sigma_i \sigma_j .
\end{equation}
Here, $i,j$ are discrete indices taking values in the index set $\{ 1 , ... , N \}$, the optimization variable $\sigma$ is a vertex of the $N$-dimensional hypercube $\Sigma = \{ \pm 1 \}^N$, and $A$ is a symmetric weight matrix with real entries. The function being minimized in \eqref{ising_problem} is sometimes referred to as the \emph{Ising Hamiltonian}\footnote{Note that in the literature on Ising machines, the convention is usually to define the Hamiltonian by $H:= - \sum J_{ij} \sigma_i \sigma_j$, i.e., with the convention $J_{ij} = -\tfrac{1}{2} A_{ij}$.} and the optimization variable $\sigma$ is sometimes called a \emph{spin configuration}. The spins $\sigma$ can be interpreted in multiple ways, such as electron spins in Ising's original formulation, graph partitions in MaxCut problems, etc.

OIMs, first introduced by \cite{TW::2019::UCNC} and \cite{TW::2021::NatComput}, have been highly successful in providing state-of-the-art (approximate) solutions to \eqref{ising_problem} in a fast, energy-efficient, and scalable manner. The idea is to first embed the above discrete problem into a continuous surrogate using \emph{Kuramoto oscillators} with \emph{second-harmonic injection locking}
\begin{equation} \label{continuous_surrogate}
	\min_{\theta \in \T^N} ~ \frac{1}{2} \sum_{i,j} A_{ij} \cos(\theta_i - \theta_j) - \frac{K_s}{2} \sum_i \cos(2 \theta_i) .
\end{equation}
Here, the optimization variable $\theta$ is now a continuous variable in the $N$-torus $\T^N$. The first term takes the same value as the Ising Hamiltonian for the choice $\sigma_i = \cos(\theta_i)$ when $\theta \in \{ 0 , \pi \}^N$, while the second term (the ``second-harmonic injection'') incentivizes $\theta$ to take values in this set. The parameter $K_s > 0$ controls the relative strength of this incentivization. With a continuous objective, one may then apply a continuous optimization algorithm such as gradient descent, which yields the dynamics
\begin{equation} \label{discrete_dynamics}
	\dot{\theta}_{i} = \sum_{j} A_{ij} \sin(\theta_{i} - \theta_{j}) - K_{s}\sin(2\theta_{i}) .
\end{equation}
Thus, for appropriately-tuned values of $K_s$, one expects that these dynamics should converge to a local minimum of \eqref{continuous_surrogate}, with $\theta \in \{ 0 , \pi \}^N$, so that a feasible solution to the original problem \eqref{ising_problem} may be recovered via the readout $\sigma_i = \cos(\theta_i)$. The great advantage of OIMs is that the dynamics \eqref{discrete_dynamics} can be implemented in physical hardware, such as FPGAs \cite{KT::2019::FPL} or optical lasers \cite{SU::2011::OptExpress}, thereby enabling the design of highly efficient analog computers to carry out these computations. A diagrammatic representation of this is illustrated in Figure~\ref{fig: OIM-pipeline}.

\begin{figure}[!h]
    \centering
\includegraphics[width=\linewidth]{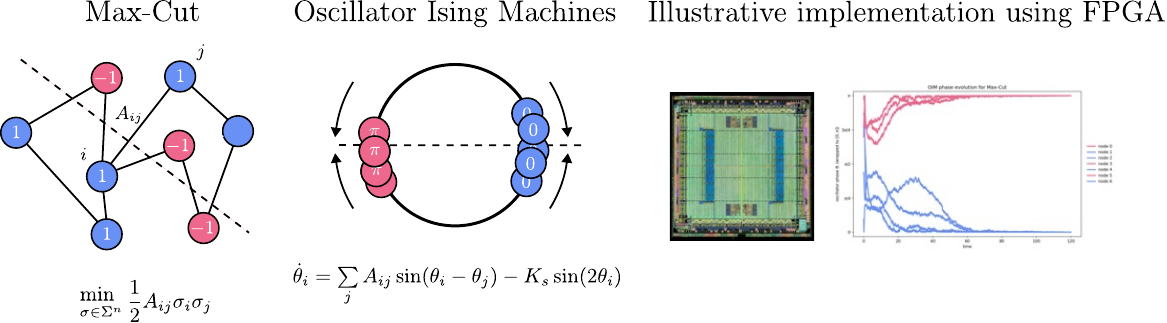}
    \caption{The above image illustrates the traditional Ising problem in a hardware implementation pipeline. The classical problem of minimizing the Ising Energy is relaxed into an optimization problem on the torus. Minimizing this surrogate optimization problem yields coupled oscillator dynamics, which are implemented on an analog machine or digital machines such as FPGA. By interpreting the readouts from such an analog implementation appropriately, we can recover approximate solutions to the original Ising problem.\protect\footnotemark[2]}
    \label{fig: OIM-pipeline}
\end{figure}
\footnotetext[2]{Image of FPGA die borrowed from Brendan G. (bgcrain03), "Xilinx Spartan XA3S250E FPGA die shot," Flickr, June 12, 2024, https://www.flickr.com/photos/196983670@N03/53786079944/. Public domain (CC PDM 1.0).}
OIMs have thus attracted considerable attention in recent years. Early experimental implementations include analog LC relaxation oscillators, VO$_2$ insulator-to-metal phase-transition relaxation oscillators, CMOS circuits, and FPGA circuits \cite{JC::2019::SciRep, SD::2021::NatElectron, WM::2022::NatElectron, KT::2019::FPL}. Empirical trade-offs and best practices have been investigated in \cite{MB::2021::IEEEAccess, AM::2021::IEDM}. These implementations have tackled a wide variety of problems, including associative memory \cite{DN::2015::JXCDC} and the maximum independent set problem \cite{AM::2020::NatCommun}, and have been scaled up to solve generalized combinatorial optimization problems on as many as 1968 nodes \cite{WM::2022::NatElectron}. Similar results have also been obtained for higher-order OIMs performing hypergraph partitioning \cite{MB::2023::SciRep}.

On the theoretical side, it is well-known that OIMs are gradient systems \cite{TW::2019::UCNC, YC::2024::Chaos, MB::2023::JApplPhys}. Works such as \cite{YC::2024::Chaos} have advanced the classification of fixed points and estimated parameter bounds for effective binarization, while \cite{MB::2023::JApplPhys} also investigates the transient dynamics before binarization. Recent works such as \cite{YC::2026::PhysRevE, AA::2025::CDC, MF::2025::arXiv} have investigated the connections between stability of the equilibrium points of OIMs and the second harmonic injection strength. Others have taken an experimental/numerical route to approaching these theoretical questions \cite{MB::2021::IEEEAccess, IA::2021::JSSC}. The existing classifications are not fully satisfactory, however. While they classify fixed points into various types, they do not identify all of the possible fixed points in a concrete manner, nor do they treat their dynamical stability. Thus, we observe a gap in the literature with interesting questions suitable for further analysis. While it is reasonable to expect a combinatorial blow-up in the number of equilibrium points, we find that it is amenable under a suitable mean-field limit.

On the other hand, mean-field analysis has proved extremely useful in understanding Kuramoto oscillator networks -- essentially OIMs without second-harmonic injection \cite{SS::2000::PhysicaD, YK::1975::LNP, JA::2005::RevModPhys, YK::2006::PTPSuppl, HC::2013::ErgodicTheory, HC::2019a::DCDS, HC::2019b::DCDS}. We thus find the investigation of mean-field OIM models to be an extremely natural direction. In particular, we emphasize the following important ideas. By adapting the existing mean-field analysis for Kuramoto oscillators, we obtain an equivalent one for OIMs. As a consequence, we can make statements about the system's behavior that are independent of $N$. We make precise statements about stability and the thresholds for binarization that no longer rely on instance-by-instance numerical trials.  

The analysis of such mean-field models is naturally carried out using the machinery of gradient flows. The recognition that many evolution PDEs can be understood as gradient flows with respect to a suitable metric -- most notably, the Wasserstein metric of optimal transport theory -- originates with the work of \cite{RJ::1998::SIAMMathAnal, FO::2001::CommPDE} and has been developed into a comprehensive theory in \cite{LA::2005::GradFlows}. This perspective has proven especially powerful for interacting-particle and mean-field systems \cite{JC::2003::RevMatIberoam, BM::2010::M3AS}, where it provides strong tools for studying well-posedness and long-time behavior. We thus adopt this viewpoint, and the model studied here is a revealing case study in this respect. We show through this problem how the Wasserstein geometry arises naturally in large-scale systems with interchangeable components. However, while elegant, this geometry also degenerates at the states we are most interested in, and we discuss the implications of this point later in the paper.

With all this in mind, we consider mean-field limits of OIMs, with the following contributions. We organize our results by the community they impact: the OIM community and the continuum modeling community. (Our results are also relevant to the optimal transport and gradient flow communities, though more as a case study than as a direct contribution to those bodies of work.)
\begin{enumerate}
\item \textbf{OIM} : 
\begin{enumerate}
\item We provide (Section \ref{sec: problem description}) formal derivations of two different mean-field models, which we term \emph{Lagrangian} and \emph{Eulerian} in analogy with fluid mechanics. We show (Section \ref{sec: gradient flow structure}) that both representations retain the gradient flow structure of the original finite-dimensional dynamics: the Lagrangian representation is a gradient flow with respect to the $L^2$ Riemannian structure on the space of phase functions. In contrast, the Eulerian representation is a gradient flow with respect to the 2-Wasserstein structure on the space of phase densities.

\item This gradient flow structure enables a classification of limit solutions (Section \ref{sec: equilibria}) and their stability (Section \ref{sec: stability}) in the totally symmetric case, using first- and second-order conditions via the calculus of variations. Ultimately, we show that all limit solutions consist of unions of fixed points with phases clustered into at most four groups, and that for almost all parameter values, only \emph{binarized} fixed points (that is, fixed points with clusters only at $0$ and/or $\pi$) can be stable. 
    \item We also provide tight bounds on the parameter values for which each binarized fixed point is stable, identifying the so-called \emph{threshold for binarization} in this model, that is, the threshold on $K/K_s$ for which nontrivial binarized states may be stable. 
    
    \item Numerical experiments (Section \ref{sec: numerical results}) confirm that these predictions accurately describe the behavior of certain large random networks, including Erd\H{o}s-R\'enyi networks.
\end{enumerate}

\item \textbf{Continuum models}: Additionally, we highlight two themes of the analysis, which seem likely to be relevant to continuum models of other large-scale systems. 
\begin{enumerate}
    \item The first theme is that we lose some regularity in the passage to the mean-field limit: while the finite-dimensional energy \eqref{continuous_surrogate} is real analytic, its infinite-dimensional counterpart is not, and the standard second order test for a local minimum is no longer available. What survives is a weaker condition which we term a \emph{directional local minimum}. This weakening is not merely technical, and we exhibit in this model a continuous family of fixed points that are directional (but not true) local minima. We discuss the implications of this for the analysis. 
    \item The second theme is that the mean-field models which we study are ultimately proxies for large but finite systems, and the extra structure available in those finite systems can be leveraged to settle some questions which the infinite-dimensional analysis leaves open.
\end{enumerate}

\end{enumerate}

We conclude (Section \ref{sec: conclusion}) with a review of the results and place them in a broader context before discussing limitations and future work.

\section{Mean-Field Models and Problem Description} \label{sec: problem description}

We introduce two mean-field models which formally describe the continuum, $N \to \infty$ limit of the model \eqref{continuous_surrogate}-\eqref{discrete_dynamics}. These models are termed \emph{Lagrangian} and \emph{Eulerian}. The Lagrangian model (Section \ref{subsec: lagrangian model}) describes the dynamics in terms of the phases of individual oscillators, and can always be defined. The Eulerian model (Section \ref{subsec: eulerian model}) describes the dynamics in terms of the density of oscillator phases, and can be defined in the special case of totally symmetric oscillators.

\subsection{Lagrangian Model} \label{subsec: lagrangian model}
 \addtocounter{footnote}{1}
The Lagrangian model is derived formally by passing from the discrete energy \eqref{continuous_surrogate} and dynamics \eqref{discrete_dynamics} to their continuum counterparts. In order for a continuum limit of either the energy \eqref{continuous_surrogate} or dynamics \eqref{discrete_dynamics} to be defined, however, a continuum limit of the underlying graphs first needs to be defined. If the underlying graphs are dense\footnote{There are indeed interesting cases where the underlying graphs are not dense and converge to different sorts of objects (as in, e.g., the case of lattices), but we do not treat these cases here.}, then the continuum limit can be described by a \emph{graphon}. Roughly, as the number of nodes $N$ grows to $\infty$, the symmetric adjacency matrices $A_N$ converge to a symmetric kernel $K: \I \times \I \to \R$, in some appropriate sense, where $\I$ denotes the continuum index set (most often $[0,1]$).

The value $K(x,y)$ may be interpreted either deterministically as a \emph{connection strength} or probabilistically as a \emph{connection probability} between the ``continuum vertices'' $x$ and $y$; in the most general setting, it represents the \emph{expected connection strength}. Regardless of interpretation, the limiting object $K$ is deterministic, with values representing a strength of interaction, and we allow it to take values in all of $\R$. In particular, $K(x,y) < 0$ corresponds to the case of \emph{attractive} oscillators (energy is minimized when the phases of $x$ and $y$ align), while $K(x,y) > 0$ corresponds to the case of \emph{repulsive} oscillators (energy is minimized when the phases of $x$ and $y$ anti-align).

With a graphon limit $K$ defined, the Lagrangian energy function and dynamics may be obtained by passing directly from the discrete indices $i,j$ to the continuum indices $x,y$ in \eqref{continuous_surrogate}-\eqref{discrete_dynamics}, replacing the adjacency matrix $A$ with the graphon $K$ (after the usual normalization $A_{ij} \sim K(x_i,x_j)/N$):
\begin{align}
	E(\theta) &= \frac{1}{2}  \iint K(x,y) \cos\big(\theta(x) - \theta(y)\big) \, \mathrm{d}x \, \mathrm{d}y - \frac{K_s}{2} \int \cos(2 \theta(x)) \, {\mathrm{d}x} , \label{lagrangian_energy} \\
	\partial_t \theta(x) &= \int K(x,y) \sin(\theta(x) - \theta(y) ) \, \mathrm{d}y - K_s \sin(2 \theta(x)) . \label{lagrangian_dynamics}
\end{align}
Here, the indices $x,y$ are now continuous variables valued in the index set $\I$. The state variable $\theta$ is now a function from the index set $\I$ to the torus $\T$ which we call the \emph{phase function}, and all integrals occur over the index set $\I$. We will show in Section \ref{sec: gradient flow structure} that the formal limits \eqref{lagrangian_energy}-\eqref{lagrangian_dynamics} retain the gradient flow structure of the original finite-dimensional system \eqref{continuous_surrogate}-\eqref{discrete_dynamics}. That is, \eqref{lagrangian_dynamics} is indeed a gradient flow of \eqref{lagrangian_energy} with respect to an appropriate metric (in this case, the $L^2$ Riemannian metric).

In this paper, we focus primarily on the \emph{totally symmetric} case where $K(x,y) \equiv K =$ constant. This case arises in many ways: deterministically as the continuum limit of complete graphs, probabilistically as the continuum limit of homogeneous random graphs, or heuristically as a natural leading-order approximation to more general graphons for which $K := \text{avg}(K(x,y))$ is taken to be the ``average interaction strength'' of the network. In this case, the model admits a second, label-independent representation -- the Eulerian model -- which we develop next.

\subsection{Eulerian Model} \label{subsec: eulerian model}

We call the $K =$ constant case ``totally symmetric'' because $K$ -- and thus the energy \eqref{lagrangian_energy} and dynamics \eqref{lagrangian_dynamics} -- are invariant under relabelings of the index set $\I$. In other words, it does not matter which labels are given to which oscillators, i.e., all oscillators are symmetric. This observation suggests that the index set is extraneous in this case, and so we seek a representation which is defined independently of the index set. Such a representation is given by the \emph{Eulerian} description, which is obtained by pushing forward the state, energy, and dynamics from the index set $\I$ to the torus $\T$.

First, the state pushes forward from a \emph{phase function} $\theta: \I \to \T$ to a \emph{phase density} $\rho$ on $\T$. Precisely, $\rho$ is the probability measure on $\T$ formed by pushing the (normalized) Lebesgue measure on $\I$ forward through the phase function $\theta$. We thus write $\rho = \theta_\# \leb$, meaning that $\rho(B) = \leb(\theta^{-1}(B))$ for all $B \subset \T$ measurable.

Two notational clarifications are in order here. First, we will abuse notation throughout this paper, using the symbol $\rho$ to refer to its ``density function'' -- even when such a density does not exist -- and using $\rho(\theta) \, \mathrm{d} \theta$ in place of the measure-theoretic $\mathrm{d} \rho$ in expressions. Second, we will overload notation, using the symbol $\theta$ both for the phase function $\theta: \I \to \T$ and for the independent spatial variable $\theta \in \T$. (This is useful, since the phase of an oscillator is then always denoted $\theta$.) Thus, we write $\rho(\theta)$ to refer to the relative density of oscillators with phase $\theta \in \T$, with the understanding that this quantity may be singular.

With this notation in place, we may then equivalently characterize the pushforward relationship $\rho = \theta_\# \leb$ via the change-of-variables formula
\begin{equation} \label{change_of_variables_eq}
    \int_\I f(\theta(x)) \, \mathrm{d}x = \int_\T f(\theta) \rho(\theta) \, \mathrm{d} \theta ,
\end{equation}
valid for all $f:\T \to \R$ measurable.

It is an important point that the transformation $\mathcal{T}: \theta \mapsto \rho = \theta_\# \leb$ is not an injective mapping. In fact, any two phase functions $\theta$ and $\phi$ which are related by a (measure-preserving) relabeling of $\I$ -- meaning that $\theta = \phi \circ \sigma$ with $\sigma_\# \leb = \leb$ -- push forward to the same phase density $\rho$. Thus, the energy \eqref{lagrangian_energy} and dynamics \eqref{lagrangian_dynamics} push forward to well-defined functions of $\rho$ exactly when they are invariant under relabelings of $\I$, i.e., exactly when $K(x,y) \equiv K = $ constant. Said another way, the Eulerian description is valid precisely when $\rho$ forms a \emph{sufficient statistic} to determine the energy and evolution of the system. This is a general feature of totally symmetric large-scale systems: the Eulerian representation is a canonical representation, and it carries a canonical geometry (as we describe in Section \ref{subsec: wasserstein gradient flows}).

In this case, the Eulerian energy function may be computed directly from \eqref{lagrangian_energy} using the change-of-variables formula \eqref{change_of_variables_eq}:
\begin{equation} \label{eulerian_energy}
    E(\rho) = \frac{K}{2} \iint \cos(\theta - \phi) \rho(\theta) \rho(\phi) \, \mathrm{d} \theta \, \mathrm{d} \phi - \frac{K_s}{2} \int \cos(2\theta) \rho(\theta) \, \mathrm{d} \theta .
\end{equation}
Keeping in line with our notation, the indices $\theta,\phi$ are now continuous variables valued in the torus $\T$, and all integrals occur over $\T$.

The dynamics \eqref{lagrangian_dynamics} in this case factor into a pointwise flow
\begin{equation} \label{flow_eqn}
    \partial_t \theta(x) = v_{\theta} (\theta(x))
\end{equation}
for a vector field
\begin{equation} \label{vector_field_theta_eqn}
    v_\theta(\bullet) := K \int_\I \sin( \bullet - \theta(y)) \, \mathrm{d}y - K_s \sin(2 \bullet) 
\end{equation}
defined on the torus $\T$. Using the change of variables formula \eqref{change_of_variables_eq} again, we may equivalently write $v$ in terms of $\rho$ as
\begin{align} \label{vector_field_rho_eqn}
    v_\rho(\bullet) &:= K \int_\T \sin(\bullet - \theta )\, \rho(\theta) \, \mathrm{d}\theta - K_s \sin(2 \bullet) \\
    & \phantom{:} = K (\sin * \rho)(\bullet) - K_s \sin(2 \bullet) , \nonumber
\end{align}
where $*$ denotes convolution. This vector field then acts to transport the phase density $\rho$ via the continuity equation, resulting in the Eulerian dynamics
\begin{align} \label{eulerian_dynamics}
    \partial_t \rho(\theta) &= - \nabla \cdot \left( \rho(\theta) \, v_\rho(\theta) \right) \\
    &= - \nabla \cdot \big( \rho(\theta) \, \left[ K (\sin * \rho) (\theta) - K_s \sin(2 \theta ) \right] \big) . \nonumber
\end{align}
Here, $\nabla \cdot$ denotes the divergence operator, and solutions are to be interpreted in an appropriate weak sense. We will show in Section \ref{sec: gradient flow structure} that the pair \eqref{eulerian_energy}, \eqref{eulerian_dynamics} also retain a gradient flow structure with respect to an appropriate metric (in this case, the 2-Wasserstein metric, which emerges naturally out of the invariance under relabeling). A pictorial representation of the push-forward relationship between Lagrangian and Eulerian representations is shown in Figure \ref{fig: Lagrangian-Eulerain Framework}.
\begin{figure}[!h]
    \centering
\includegraphics[width=\linewidth]{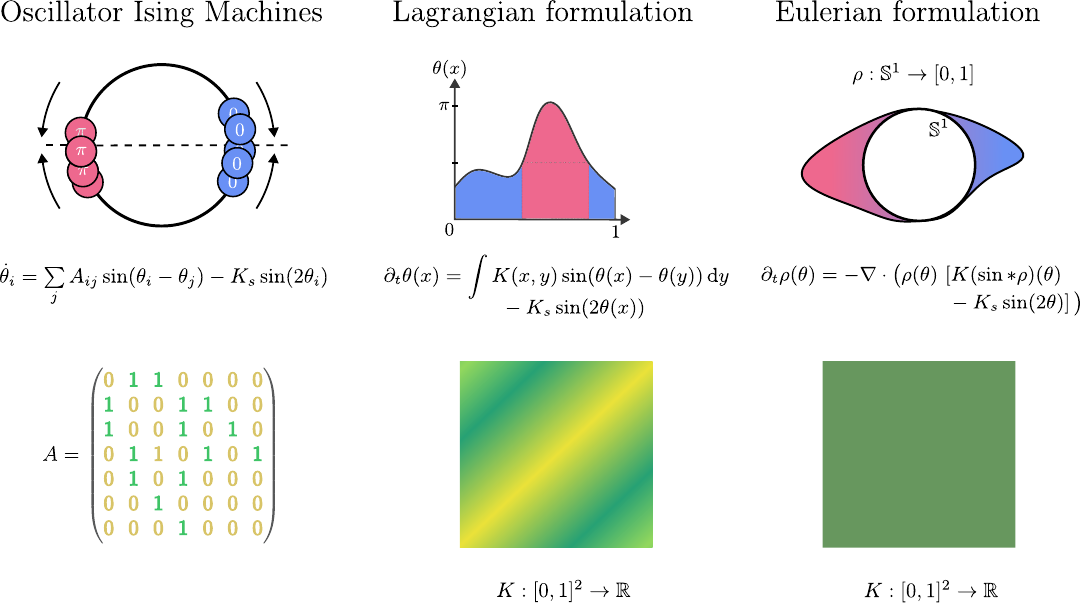}
    \caption{%
    The above image illustrates the shift in perspective from the finite size OIM, to the Lagrangian formulation and finally the Eulerian formulation. The first row depicts the visualization of the finite-OIM, Lagrangian and Eulerian formulation as a network of oscillators, the phase function on $[0,1]$ and density function on the circle. The second row shows the corresponding state dynamics. The third row illustrates the equivalent "adjacency matrix" in the three formulations. The traditional graph adjacency matrix in the finite-OIM, the symmetric graphon limit in the Lagrangian formulation and a constant graphon in the Eulerian case.}
    \label{fig: Lagrangian-Eulerain Framework}
\end{figure}

We may now state the central problem which we study in this work.

\begin{problem} \label{main_prob}
    Classify the limiting (i.e.~$t \to \infty$) behavior of solutions to \eqref{lagrangian_dynamics} for $K(x,y) \equiv K =$ constant (equivalently, of solutions to \eqref{eulerian_dynamics}).
\end{problem}

\section{Gradient Flow Structure and Limiting Behavior} \label{sec: gradient flow structure}

In this section, we establish our main theoretical tools to address Problem \ref{main_prob}. The central observation is that both mean-field models introduced in the previous section are gradient flows. This gradient flow structure strongly constrains the behavior of solutions, and reduces the classification of limiting behavior to a classification of equilibria -- a program which we carry out systematically in Sections \ref{sec: equilibria} and \ref{sec: stability}. We begin (Section \ref{subsec: gradient flow formulation}) by establishing the gradient flow structure of the Lagrangian model with respect to the $L^2$ Riemannian metric on the space of phase functions. This structure provides (Section \ref{subsec: fixed points and stability}) characterizations of fixed points and their stability via the first and second variations of the energy function -- although, as we discuss at length, the second-order conditions available here are weaker than one might expect. We then describe (Section \ref{subsec: constant k case}) the additional structure that emerges in the totally symmetric case, and use it (Section \ref{subsec: limiting behavior}) to show that the set of accumulation points of every trajectory is a connected union of fixed points contained in a single energy level. We also discuss what can and cannot be concluded about the convergence of trajectories to individual fixed points. We close (Section \ref{subsec: wasserstein gradient flows}) by discussing the gradient flow structure of the Eulerian model with respect to the 2-Wasserstein metric on the space of phase densities, the sense in which this geometry is canonical for systems of this type, and its limitations in the problem at hand.

\subsection{Gradient Flow Formulation} \label{subsec: gradient flow formulation}

When one says that some equation is a gradient flow, it always needs to be specified which geometry the gradient flow is with respect to, since the gradient (and thus the dynamics) depends on the choice of metric. In the finite-dimensional Euclidean setting, the gradient of a function $f$ is the vector of partial derivatives $\nabla f = (\partial_{x_1} f , ... , \partial_{x_n} f)$. In more general settings, one instead defines the gradient as the Riesz representation of the first variation
\begin{equation} \label{first_variation_eq}
    Df(x)[v] := \lim_{h \to 0} \frac{f(x + hv) - f(x)}{h} ,
\end{equation}
that is (supposing that this operator is bounded), via the identity
\begin{equation} \label{gradient_eq}
    \langle \nabla f(x) , v \rangle = Df(x)[v] \qquad \forall v .
\end{equation}
On a Riemannian manifold one uses the metric $\langle \cdot , \cdot \rangle_x$ in place of a fixed inner product.

The state space in the Lagrangian mean-field model is given by the set of phase functions $\Theta := L^2(\I;\T)$. This space is neither finite-dimensional nor a vector space. Rather, it has the structure of an infinite-dimensional Riemannian manifold, sometimes called a \emph{Hilbert} manifold. Concretely, $\Theta$ may be realized as the quotient of $L^2(\I;\R)$ obtained by identifying values modulo $2\pi$, and the corresponding quotient map is a local isometry. The tangent spaces are therefore all identified with $L^2(\I;\R)$, and inherit the standard $L^2$ inner product $\langle \cdot , \cdot \rangle$. Importantly, this also means that $\Theta$ is flat, and so for our purposes, we may simply work in local coordinates, taking derivatives in $\Theta$ just as we would in $L^2(\I;\R)$.

Thus, we would like to write the gradient flow as
\begin{equation} \label{gradient_flow_eq}
    \partial_t \theta = - \nabla E (\theta) ,
\end{equation}
taking the gradient to be defined as above. In order to do this rigorously, however, we first need to verify the differentiability of $E$. This is handled by the following proposition.

\begin{proposition} \label{regularity_prop}
    Supposing that $K$ is uniformly bounded, the energy function \eqref{lagrangian_energy} is (Fr\'echet) differentiable on $\Theta$, with gradient
    \begin{equation} \label{lagrangian_gradient_eq}
        \nabla E(\theta)(x) = - \int K(x,y) \sin(\theta(x) - \theta(y)) \, \mathrm{d}y + K_s \sin(2\theta(x)) .
    \end{equation}
    Furthermore, the gradient map $\theta \mapsto \nabla E(\theta)$ is bounded and globally Lipschitz on $\Theta$.
\end{proposition}

\begin{proof}
    See Appendix \ref{regularity_prop_proof}.
\end{proof}

This level of regularity\footnote{It is worth noting that if we take the $L^\infty$ structure on $\Theta$ rather than the $L^2$ structure, the energy function $E$ is smooth, and in fact, analytic. However, the price of this analytic structure is that we lose the geometric structure, and with it, the gradient flow.} ($E$ differentiable with Lipschitz derivative) may be summarized by saying that $E$ belongs to the regularity class $C^{1,1}$. In this case\footnote{Much recent effort has been devoted to studying gradient flows both in less structured spaces (e.g. metric spaces \cite{LA::2005::GradFlows}), and with less regular energy functions. Our setting (Hilbert manifold state space with $C^{1,1}$ energy) is relatively ``tame'' in this respect.}, we may write the gradient flow as \eqref{gradient_flow_eq} with no ambiguity as to what is meant. This immediately allows us to verify that the Lagrangian dynamics \eqref{lagrangian_dynamics} are indeed the gradient flow dynamics of the Lagrangian energy \eqref{lagrangian_energy}, and moreover that these dynamics are well-posed.

\begin{corollary} \label{gradient_flow_cor}
    The dynamics \eqref{lagrangian_dynamics} are the gradient flow dynamics of the energy \eqref{lagrangian_energy} with respect to the $L^2$ Riemannian structure on $\Theta$, and are globally well-posed.
\end{corollary}

\begin{proof}
    The first claim follows directly from Proposition \ref{regularity_prop}. Since $\nabla E$ is globally Lipschitz, the Picard--Lindel\"of theorem applies to give global well-posedness. (Note that since $\Theta$ is complete and boundaryless, solutions cannot escape in finite time.)
\end{proof}

\subsection{Fixed Points and Stability Conditions} \label{subsec: fixed points and stability}

The gradient flow structure implies that the energy $E$ is a Lyapunov function for its own gradient flow. This places strong constraints on the behavior of trajectories and allows us to characterize fixed points and their stability in terms of the Riemannian gradient $\nabla E$ and Hessian $\nabla^2 E$. However, the treatment of stability is somewhat delicate in our setting, and we spend some time on this issue and its implications.

As a direct consequence of the characterization \eqref{gradient_flow_eq}, fixed points of the dynamics \eqref{lagrangian_dynamics} may be identified with stationary points of the energy \eqref{lagrangian_energy}.

\begin{lemma} \label{fixed_point_lemma}
    A point $\theta \in \Theta$ is a fixed point of the gradient flow dynamics \eqref{lagrangian_dynamics} if and only if $\nabla E(\theta)=0$.
\end{lemma}

\begin{proof}
    Follows immediately from the characterization \eqref{gradient_flow_eq}.
\end{proof}

The characterization of stability of these fixed points then leverages the fact that $E$ is a Lyapunov function for its own gradient flow, with trajectories satisfying the following \emph{energy dissipation identity}.

\begin{lemma} \label{ede_lem}
    Let $\theta(t)$ denote any trajectory of the dynamics \eqref{lagrangian_dynamics}. Then $E(\theta(t))$ satisfies the following \emph{energy dissipation identity}:
    \begin{equation} \label{energy_dissipation_identity}
        \frac{d}{dt} E(\theta(t)) = -\| \nabla E(\theta(t)) \|_{L^2}^2 .
    \end{equation}
    As a consequence, $E(\theta(t))$ is monotone nonincreasing.
\end{lemma}

\begin{proof}
    Using the chain rule and the characterizations \eqref{gradient_eq} and \eqref{gradient_flow_eq}, we have
    \begin{multline}
        \frac{d}{dt} E(\theta(t)) = DE(\theta(t))[\partial_t \theta(t)] = \langle \nabla E (\theta(t)) , \partial_t \theta(t) \rangle \\
        = \langle \nabla E (\theta(t)) , - \nabla E (\theta(t)) \rangle = -\| \nabla E (\theta(t)) \|_{L^2}^2 .
    \end{multline}
    Since this quantity is nonpositive, $E(\theta(t))$ is monotone nonincreasing.
\end{proof}

With a Lyapunov function in hand, one then usually hopes to conclude that (1) fixed points are asymptotically stable if and only if they are local minima of $E$, (2) local minima can be characterized by the positivity of the Hessian $\nabla^2E$, and (3) ``most'' trajectories of a gradient flow limit to asymptotically stable fixed points -- together allowing one to characterize generic limiting behavior in terms of the Hessian. In finite dimensions with a $C^2$ energy function, this reasoning goes through more or less exactly; in our setting, it does not. The rest of this section is devoted to unpacking what survives of (1) and (2); the point (3) is deferred to Section \ref{subsec: limiting behavior}.

The main issue is that for an energy function $E$ which is merely $C^{1,1}$, the Hessian map $\theta \mapsto \nabla^2 E(\theta)$ is bounded but not continuous. This blocks one from using the usual second-order Taylor expansion argument to conclude that a stationary point $\theta$ is a local minimum (thus asymptotically stable) from the positivity of the Hessian $\nabla^2 E (\theta)$. However, while the Hessian map fails to be continuous in the usual (i.e.~\emph{operator norm}) topology, it is continuous in a weaker topology known as the \emph{strong operator topology}. In short, the strong operator topology considers convergence of operators along fixed directions and is equivalent to the norm topology when this convergence is uniform in direction. (This is analogous to the difference between G\^ateaux and Fr\'echet differentials in the calculus of variations.) The upshot of this is that we may carry out the Taylor expansion arguments in this weaker setting, working with a weaker notion of \emph{directional local minimum}, which is necessary (but not sufficient) for a true local minimum. This yields the stability conditions which we apply later in Section \ref{sec: stability}. We explain this all as follows.

First, the Riemannian Hessian $\nabla^2 E$ is defined in the same manner as the Riemannian gradient, i.e., as the Riesz representation of the second variation. We first define the second variation as the symmetric bilinear operator
\begin{equation} \label{second_variation_eq}
    D^2E(\theta)[p,q] := \lim_{h,\ell \to 0} \frac{E(\theta + hp + \ell q) - E(\theta+hp) - E(\theta+\ell q) + E(\theta)}{h \ell} ,
\end{equation}
and then (supposing boundedness), the Hessian is defined via the identity
\begin{equation}
    D^2E(\theta)[p,q] = \langle p , \nabla^2 E(\theta) \, q \rangle \qquad \forall p,q .
\end{equation}

Using this definition, we now compute and verify the regularity of the Hessian.

\begin{proposition} \label{regularity_prop_2}
    Supposing that $K$ is uniformly bounded, the gradient map $\theta \mapsto \nabla E(\theta)$ is (G\^ateaux) differentiable on $\Theta$, with derivative
    \begin{equation} \label{hessian_splitting_eq}
        \nabla^2 E(\theta) ~=~ M_{a(\theta)} + T_\theta ,
    \end{equation}
    where $M_{a(\theta)}$ denotes multiplication by the function
    \begin{equation}
        a(\theta)(x) ~:=~ 2 K_s \cos(2\theta(x)) - \int K(x,y) \cos\big(\theta(x)-\theta(y)\big) \, \mathrm{d}y
    \end{equation}
    and $T_\theta$ denotes the integral operator with kernel $K(x,y) \cos(\theta(x)-\theta(y))$. Furthermore, the Hessian map $\theta \mapsto \nabla^2 E(\theta)$ is bounded and continuous on $\Theta$ in the strong operator topology.
\end{proposition}

\begin{proof}
    See Appendix \ref{regularity_prop_2_appen}.
\end{proof}

We now introduce the notion of \emph{directional local minimum} -- the appropriate relaxation of the notion of local minimum to this weaker setting.

\begin{definition}[Directional Local Minimum]
    We say that a point $\theta \in \Theta$ is a \emph{directional local minimum} of $E$ if for every $p \in L^2$ there exists $\varepsilon_p > 0$ such that $E(\theta + hp) \geq E(\theta)$ for all $|h| < \varepsilon_p$. We say that it is a \emph{strong directional local minimum} if, moreover, there exists $\alpha > 0$, independent of $p$, such that $E(\theta + hp) \geq E(\theta) + \alpha \, h^2 \| p \|_{L^2}^2$ for every $p \in L^2$ and $|h| < \varepsilon_p$.
\end{definition}

\begin{lemma} \label{stability_lemma}
    Let $\theta$ be a stationary point of the energy \eqref{lagrangian_energy}. Then:
    \begin{enumerate}[label=\alph*)]
        \item For $\theta$ to be a directional local minimum of \eqref{lagrangian_energy}, it is necessary that \\ $\langle p , \nabla^2 E(\theta) \, p \rangle \geq 0$ for all $p \in L^2$.
        \item For $\theta$ to be a strong directional local minimum of \eqref{lagrangian_energy}, it is sufficient that there exist $c > 0$ such that $\langle p , \nabla^2E(\theta) \, p \rangle \geq c \, \| p \|_{L^2}^2 $ for all $p \in L^2$.
    \end{enumerate}
\end{lemma}

\begin{proof}
    First, note that for each $p \in L^2$, the function $h \mapsto E(\theta + hp)$ is twice continuously differentiable by Proposition \ref{regularity_prop_2}, with second derivative $D^2E(\theta + hp)[p,p]$ depending continuously on $h$. We may then Taylor expand to obtain
    \begin{equation}
        E(\theta + h p) = E(\theta) + h DE(\theta)[p] + \frac{h^2}{2} D^2E (\theta) [p,p] + o(h^2) .
    \end{equation}
    The linear term vanishes since $\theta$ is a fixed point, and therefore, for sufficiently small $h$, the quadratic term dominates. Thus, if there exists $p$ such that $\langle p , \nabla^2 E(\theta) \, p \rangle < 0$, then $E(\theta+hp) < E(\theta)$ for small $h$, and $\theta$ cannot be a directional local minimum. On the other hand, suppose that $\langle p , \nabla^2 E(\theta) \, p \rangle \geq c \, \| p \|^2$. The remainder term satisfies $|o(h^2)| \leq \tfrac{c}{4} h^2 \| p \|^2$ for all $|h| < \varepsilon_p$, where the threshold $\varepsilon_p$ depends on the direction $p$, and therefore $E(\theta + hp) \geq E(\theta) + \tfrac{c}{4} h^2 \, \| p \|^2$ for all $|h| < \varepsilon_p$. Since the constant $\tfrac{c}{4}$ here is independent of $p$, if this holds for all $p \in L^2$, then $\theta$ is a strong directional local minimum.
\end{proof}

The notion of directional local minimum -- being a weaker notion than true local minimum -- is necessary (but not in general sufficient) for asymptotic stability.

\begin{lemma}\label{lem: asym_to_dlm}
    Let $\theta$ be a fixed point of the dynamics \eqref{lagrangian_dynamics}. Then for $\theta$ to be asymptotically stable, it is necessary that $\theta$ be a directional local minimum of \eqref{lagrangian_energy}.
\end{lemma}

\begin{proof}
    By contrapositive, suppose that $\theta$ is not a directional local minimum. Then by definition there exists a direction $p$ for which $E(\theta+hp) < E(\theta)$ for arbitrarily small $h$. Denote by $\theta_h(t)$ the trajectory of \eqref{lagrangian_dynamics} starting from $\theta+h p$. The energy dissipation identity \eqref{energy_dissipation_identity} ensures that $E(\theta_h(t))$ is nonincreasing, thus by the continuity of $E$, $\theta_h(t)$ remains bounded away from $\theta$. Since $h$ may be taken arbitrarily small, $\theta$ cannot be asymptotically stable.
\end{proof}

Some remarks regarding this notion of directional local minimum are now in order. First, we emphasize again that in the infinite-dimensional setting, directional local minimum is a strictly weaker notion than true local minimum. (Indeed, the system that we study in this paper provides an interesting case study in this respect -- we will show in Section \ref{sec: stability} that the \emph{binarized} family of fixed points forms a continuous family which are strong directional (but not true) local minima.) The gap between these two notions is exactly due to the failure of the Hessian map $\theta \mapsto \nabla^2 E(\theta)$ to be continuous in the norm topology. In particular, in the definition of strong directional local minimum, note that the \emph{rate} $\alpha$ is uniform over directions $p$ (i.e.~the Hessian is \emph{uniformly positive definite}), while the \emph{neighborhood} $\varepsilon_p$ is not.

Now, in finite dimensions, a strong directional local minimum is necessarily a true local minimum, since the compactness of the unit sphere allows one to obtain the requisite uniformity automatically. Thus, in finite dimensions, both notions serve equally well to characterize the asymptotic stability of fixed points. However, in the infinite-dimensional model which we study in this paper, the notion of directional local minimum is well-behaved, while the notion of true local minimum is not. In particular, it is the directional local minima (and the corresponding positivity of the Hessian) which provides the correct infinite-dimensional proxy for the behavior of the finite-dimensional models which we study. This point is important enough that it is worth reiterating in the following remark.

\begin{remark} \label{finite_correspondence_rmk}
    The mean-field models \eqref{lagrangian_energy}-\eqref{lagrangian_dynamics} and \eqref{eulerian_energy}-\eqref{eulerian_dynamics} that we study, while themselves infinite-dimensional, are really proxies for large but \emph{finite} systems (whether real or simulated). Often, we are not as interested in the behavior of the infinite-dimensional models \emph{themselves} so much as in how that behavior translates to the finite-dimensional approximations that we care about.
\end{remark}

The above observation cuts both ways. On one hand, finite-dimensional systems and their infinite-dimensional counterparts can exhibit different properties, and transferring conclusions from one to the other can be a delicate issue. On the other hand, finite-dimensional systems often enjoy structure -- compactness, real-analyticity -- that their infinite-dimensional counterparts lack, and we can get a lot of leverage out of this since it is usually the finite-dimensional systems that we actually care about. A perfect example is as above: while positivity of the Hessian does \emph{not} imply asymptotic stability in the infinite-dimensional model, it is still the condition which explains the behavior of OIMs \emph{as observed in practice}. (This is evidenced by our numerical results in Section \ref{sec: numerical results}.) We will also return to this same sort of argument when discussing convergence in Section \ref{subsec: limiting behavior}.

\subsection{Special Structure in the Totally Symmetric Case} \label{subsec: constant k case}

In the totally symmetric case where $K(x,y) \equiv K =$ constant, there is a great deal of additional structure in the problem. Recall from Section \ref{sec: problem description} that in this case, the invariance under relabeling renders the index set extraneous. We may thus choose\footnote{This may be justified by the Borel isomorphism theorem, for instance.} it to be whatever is convenient for our analysis, say, $[0,1]$. Additionally, we may pick the initial state $\theta_0$ to have nice structure on this set. For example, by the monotone rearrangement theorem, we may pick $\theta_0$ to be monotone nondecreasing in our chosen coordinates on $\T$, say, $[0,2\pi)$.

Recall also that in this case, the dynamics factor into a pointwise flow of the form $\partial_t \theta(x) = v_\theta (\theta(x))$ for a vector field $v_\theta$ defined on the torus $\T$. The comparison principle ensures that the flow of this vector field preserves the ordering of integral curves, and thus the monotonicity of $\theta$. There is a slight technicality in how the evolution interacts with the choice of coordinate system, which is handled by the following definition and lemma.

\begin{definition}[monotonicity] \label{monotonicity_def}
    We say that a phase function $\theta:[0,1] \to \T$ is \emph{monotone} if there exists a coordinate system $\psi: \T \to [a,a+2\pi)$ in which $\psi \circ \theta$ is monotone in the usual sense: $x \leq y ~ \Rightarrow ~ \psi(\theta(x)) \leq \psi(\theta(y))$.
\end{definition}

\begin{lemma} \label{monotonicity_lemma}
    Let $\theta(t)$ denote the solution of \eqref{flow_eqn}-\eqref{vector_field_theta_eqn} starting from initial condition $\theta_0$. If $\theta_0$ is monotone (in the sense of Definition \ref{monotonicity_def}), then $\theta(t)$ is monotone for all $t$.
\end{lemma}

\begin{proof}
    Let $\Phi_t: \T \to \T$ denote the flow map of $v_{\theta(t)}$. (Note that $v_\theta$ is smooth and uniformly bounded, so $\Phi$ is well-defined.) Supposing that $\theta_0$ is monotone in the coordinate system $\psi: \T \to [a,a+2\pi)$, then by the comparison principle and continuity of $\Phi$, $\theta(t)$ is monotone in the coordinate system $\psi': \T \to [b,b+2\pi)$, where $b := \psi \circ \Phi_t \circ \psi^{-1}(a)$.
\end{proof}

We emphasize that the above definition of monotonicity states that there \emph{exists} a coordinate system in which $\theta$ is monotone -- not that $\theta$ is necessarily monotone in the \emph{chosen} coordinate system $[0,2\pi)$. While the initial state $\theta_0$ can always be chosen to be monotone in this coordinate system, $\theta(t)$ may evolve so that it ``wraps'' around from $2\pi$ to $0$ somewhere in the interior of the interval $[0,1]$. This is depicted in Figure \ref{fig:wrapping}. That said, $\theta(t)$ can ``wrap'' a maximum of one time, and therefore, in any coordinate system, a monotone phase function $\theta$ has total variation bounded by $4 \pi$. We will use this fact later in Section \ref{subsec: limiting behavior} to prove the precompactness of trajectories.

\begin{figure}[!h]
    \centering
    \includegraphics[width=0.8\linewidth]{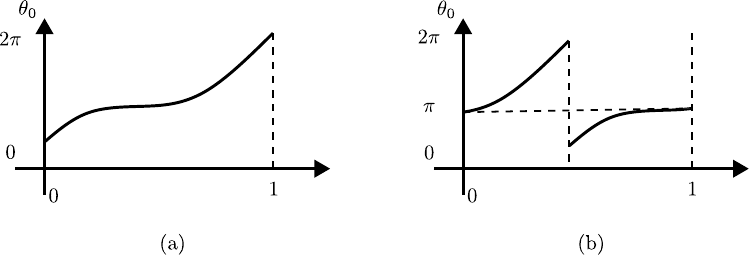}
    \caption{Illustration of a monotone phase function (a) and wrapping in coordinates (b). Note that the definition of monotonicity implies that at most one wrapping can occur.}
    \label{fig:wrapping}
\end{figure}
Additionally, the energy function $E(\theta)$ and vector field $v_\theta$ can be finitely parameterized in terms of the \emph{order parameters} of the system. (In fact, the vector field depends on the first order parameter alone.) That is, defining the first and second order parameters by
\begin{equation} \label{order_parameters_eqn}
    z_1 := \int_0^1 e^{i \theta(x)} \, \mathrm{d}x , \qquad \qquad  z_2 := \int_0^1 e^{2i\theta(x)} \mathrm{d}x ,
\end{equation}
the energy function $E(\theta)$ and vector field $v_\theta$ admit the parameterizations
\begin{align}
    E(z_1,z_2) &= \frac{K}{2} |z_1|^2 - \frac{K_s}{2} \text{Re}(z_2) , \label{order_param_e}\\
    v_{z_1} &= -K \, \text{Im}(z_1 e^{-i \bullet }) - K_s \sin(2 \bullet) . \label{order_param_v}
\end{align}

The order parameters $z_1$ and $z_2$ are complex parameters, but can be given concrete physical interpretations. Using the change-of-variables formula \eqref{change_of_variables_eq}, the order parameters can equivalently be written in terms of $\rho$ as
\begin{equation} \label{order_parameters_rho_eqn}
    z_1 := \int_\T e^{i \theta} \rho(\theta) \, \mathrm{d}\theta , \qquad \qquad  z_2 := \int_\T e^{2i\theta} \rho(\theta) \, \mathrm{d}\theta .
\end{equation}
By identifying $\T$ with the unit circle in the complex plane and interpreting $\rho$ as a mass distribution on $\T$, the first order parameter $z_1$ represents the complex coordinate of the center of mass of $\rho$, while $z_2$ is the second harmonic of $\rho$, whose real part measures the concentration of $\rho$ around the $\{0,\pi\}$-axis.

The first order parameter $z_1$ also plays the role of a ``mean-field parameter'', in the sense that the evolution of $\theta$ (equivalently, $\rho$) factors into the evolution of a collection of oscillators which are coupled only through their interaction with the parameter $z_1$. This mean-field coupling through $z_1$ underlies the conditional convergence we develop in Section \ref{subsec: limiting behavior} as well as the classification of fixed points in Section \ref{sec: equilibria}.

\subsection{Limiting Behavior} \label{subsec: limiting behavior}

We now investigate the long-time behavior of trajectories in the totally symmetric case. First, as a consequence of the energy dissipation identity, both $E(\theta(t))$ and $\nabla E(\theta(t))$ converge.

\begin{lemma} \label{energy_convergence_lem}
    Let $\theta(t)$ denote any trajectory of the dynamics \eqref{lagrangian_dynamics}. Then the quantity $E(\theta(t))$ converges. Furthermore, the gradient $\nabla E(\theta(t)) \to 0$ in $L^2$.
\end{lemma}

\begin{proof}
    By Lemma \ref{ede_lem}, the quantity $E(\theta(t))$ is monotone nonincreasing. Since $E$ is also bounded below, $E(\theta(t))$ must converge. Then by the energy dissipation identity \eqref{energy_dissipation_identity}, $\| \nabla E(\theta(t)) \|_{L^2}^2$ must be integrable in $t$. Moreover, since $\nabla E$ is bounded and Lipschitz (Proposition \ref{regularity_prop}) and $\partial_t \theta = -\nabla E(\theta)$, the map $t \mapsto \| \nabla E(\theta(t)) \|_{L^2}^2$ is Lipschitz. An integrable, uniformly continuous function must vanish in the limit, so $\| \nabla E(\theta(t)) \|_{L^2} \to 0$.
\end{proof}

Second, the total variation bound on $\theta$ given in Section \ref{subsec: constant k case} allows us to prove the precompactness of trajectories via Helly's selection principle.

\begin{lemma} \label{precompactness_lem}
    Each trajectory $\{ \theta(t) \}_{t \in [0,\infty)}$ of the system \eqref{flow_eqn}-\eqref{vector_field_theta_eqn} has compact closure in $\Theta$.
\end{lemma}

\begin{proof}
    Following the discussion in Section \ref{subsec: constant k case}, we may assume $\theta_0$ to be monotone. Fix the coordinates $[0,2\pi)$ on $\T$. Then $\theta(t)$ is uniformly bounded in these coordinates, and as a consequence of Lemma \ref{monotonicity_lemma}, has total variation bounded by $4\pi$. Then by Helly's selection principle \cite[Section 36.5]{AK::1975::Book}, every sequence in $\{ \theta(t) \}_{t \in [0,\infty)}$ admits a pointwise convergent subsequence. Since these functions are uniformly bounded, the dominated convergence theorem upgrades this to convergence in $L^2([0,1];\R)$. Thus, every sequence in the trajectory admits an $L^2$-convergent subsequence, so the trajectory has compact closure in $\Theta$.
\end{proof}

With precompactness in hand, we may now apply the LaSalle-Krasovskii invariance principle.

\begin{lemma} \label{accumulation_lem}
    The set of accumulation points of each trajectory $\theta$ of \eqref{flow_eqn}-\eqref{vector_field_theta_eqn} consists of a nonempty, compact, connected union of fixed points contained within a single energy level of $E$.
\end{lemma}

\begin{proof}
    By Lemma \ref{precompactness_lem}, the trajectory $\{ \theta(t) \}_{t \in [0,\infty)}$ has compact closure in $\Theta$, and thus a nonempty set of accumulation points. By the LaSalle-Krasovskii invariance principle and Lemma \ref{energy_convergence_lem}, this set must be compact, connected, and contained in the union of complete trajectories contained entirely within the set $\{ E(\theta) = E^* \} \cap \{ \nabla E (\theta) = 0 \}$, which is exactly the set of fixed points contained within the energy level $E^* := \lim_{t \to \infty} E(\theta(t))$.
\end{proof}

Lemma \ref{accumulation_lem} shows that the accumulation set of each trajectory is a connected union of fixed points lying in a single energy level. This does not on its own, however, imply either that (1) every trajectory converges to an individual fixed point, or that (2) ``most'' trajectories converge to stable fixed points.

There are two standard routes to proving (1). Route (a) is to show that the fixed points in the accumulation set are isolated; since the accumulation set is connected, it then consists of a single fixed point, and since trajectories are precompact, the trajectory converges to it. Route (b) is to show that $E$ satisfies a \L{}ojasiewicz inequality, from which it follows that trajectories have finite length and thus converge. The main route to proving (2) is stable manifold theory: one shows that the stable manifolds of the unstable fixed points have positive codimension, so that their union is negligible.

In the infinite-dimensional model, both routes to (1) fail. Route (a) fails because the fixed points are not isolated: the parameterization given in Section \ref{sec: equilibria} shows that they come in continuous families. Route (b) fails for exactly the same reason that we had to be careful about directional local minima in Section \ref{subsec: fixed points and stability}: the binarized family exhibited in Section \ref{sec: stability} is a continuous family of fixed points along which $E$ is non-constant, which precludes any \L{}ojasiewicz inequality. Stable manifold theory, meanwhile, is delicate and technical in infinite dimensions, and this sort of development is outside the scope of this work. We therefore leave both (1) and (2) open for the mean-field models themselves.

However, for the purposes of predicting behavior in the large finite-dimensional models that we care about, the insight of Remark \ref{finite_correspondence_rmk} can be brought to bear. The finite-dimensional energy \eqref{continuous_surrogate} is real-analytic and therefore satisfies a \L{}ojasiewicz inequality, which gives (1) by route (b). Analyticity also ensures that the fixed points comprise finitely many connected components, and the unstable ones among them have stable manifolds of positive codimension \cite[Theorem~3.2.1]{SW::2003::Springer}; the union of these therefore has measure zero, so that almost every trajectory converges to a local minimum, giving (2). This is again evidenced in our numerical results in Section \ref{sec: numerical results}.

\subsection{Gradient Flows in the Eulerian Setting} \label{subsec: wasserstein gradient flows}

Having discussed the gradient flow structure of the Lagrangian model at length, we now discuss the gradient flow structure of the Eulerian model \eqref{eulerian_energy}-\eqref{eulerian_dynamics} and its relation with material presented in Sections \ref{sec: problem description} and \ref{sec: gradient flow structure}. The Eulerian model is a gradient flow with respect to the 2-Wasserstein metric on the space of phase densities -- a fact that emerges naturally from the quotient structure of this space. Although an elegant approach, the Eulerian structure presents certain difficulties in our setting: the geometry degenerates precisely at the fixed points that we care about, and provides no second-order stability calculus there. We develop both the construction and its limitations below.

Recall from Section \ref{sec: problem description} that the phase density $\rho$ is related to the phase function $\theta$ by the transformation $\mathcal{T}: \theta \mapsto \rho = \theta_\# \leb$, and that this transformation is many-to-one: any two phase functions related by a measure-preserving relabeling of $\I$ push forward to the same $\rho$. Each phase density is thus identified with an entire equivalence class $[\theta]$ of phase functions, equivalent modulo relabeling, and the selection of a monotone representative (as made in Section \ref{subsec: constant k case}) is simply one way of selecting a representative from each class. Recall further that the energy and dynamics push forward to well-defined functions of $\rho$ precisely when they are constant on each such class. When they are, one expects the Eulerian representation to retain a gradient flow structure as well -- and, since a gradient flow must always be defined with respect to some metric, one expects the relevant metric to be the quotient of the $L^2$ Riemannian metric on $\Theta$. This is indeed the case, and this quotient metric turns out to be exactly the 2-Wasserstein metric on the space of phase densities.

This is a general construction: starting from a space $\Omega$ with Riemannian distance $d$, consider the space $\mathcal{M}$ of maps from an index set $\I$ into $\Omega$, equipped with the $L^2$ distance induced by $d$. Quotienting $\mathcal{M}$ by relabelings of $\I$ yields the space of probability densities on $\Omega$, and the induced quotient metric is precisely the 2-Wasserstein distance. The Wasserstein geometry is therefore a canonical geometry arising in collective systems with totally symmetric components, and this explains its appearance in this model in addition to many other physical and statistical systems.

Otto \cite{FO::2001::CommPDE} (building on the work of Brenier \cite{YB::1991::CPAM}) was the first to recognize this quotient metric structure in the Wasserstein distance. Working formally in the category of infinite-dimensional Riemannian manifolds, he showed that the transformation $\mathcal{T}: \theta \mapsto \rho = \theta_\# \leb$ is \emph{almost} a Riemannian submersion. It has the correct formal structure, but the target space (i.e., the Wasserstein space of probability densities over $\Omega$) is not quite a Riemannian manifold, since it has a ``boundary'' -- the measures which are not absolutely continuous -- where the induced metric degenerates.

The degeneracy at the boundary presents challenges both for the rigorous definition of gradient flow in this setting and the stability analysis, since the gradient and Hessian can no longer be defined via the identities $\langle \nabla f(x),v \rangle_x = Df(x)[v]$ or $\langle u , \nabla^2 f(x) \, v\,  \rangle_x = D^2f(x)[u,v]$ at the boundary. For the most part, there are good tools for dealing with this. Option (a) is to work on the interior of the Wasserstein space, where the formal Riemannian structure can usually be trusted; option (b) is to treat the dynamics as a gradient flow over a metric space, extending the definitions of gradient and gradient flow appropriately. The book by Ambrosio, Gigli, and Savaré \cite{LA::2005::GradFlows} puts both of these approaches on a rigorous footing.

With these tools, most of the results established above for the Lagrangian model carry over to the Eulerian model as well: well-posedness and energy dissipation may be established by alternative arguments, precompactness of trajectories follows from Prokhorov rather than Helly, and the characterization of fixed points takes on the more natural form given in Section \ref{sec: equilibria}. The main point of deviation is the stability analysis, where both of the above options fail. Option (a) fails because the fixed points that we care about are always at the boundary -- indeed, we show in Section \ref{sec: equilibria} that fixed points are always atomic, and working with the formal Riemannian structure at these points gives a \emph{genuinely} incorrect answer for their stability. Option (b) fails as well, though more subtly: the metric-space theory delivers nice theoretical tools, but it provides no second-order stability calculus at boundary points, i.e., there is no analogue of the Lagrangian Hessian with which to actually compute stability. We return to both of these points at the beginning of Section \ref{sec: stability}; for now, we simply remark that this is the primary reason that we carry out most of the development in the Lagrangian setting.

\section{Analysis of Equilibria} \label{sec: equilibria}

So far we have looked at the structure of the gradient flow system; now we look at its equilibrium points. It is more natural to carry out the fixed point analysis in the Eulerian setting, so we first provide an analogue of Lemma~\ref{fixed_point_lemma} regarding fixed point conditions for $\rho$.

\begin{corollary}
    A point $\rho \in P(\mathbb{T})$ is a fixed point of the system \eqref{eulerian_dynamics} if and only if $\rho$ is supported on the zero set of the velocity field
    \begin{equation} \label{v_rho_eqn}
        v_\rho(\bullet) = K \int_\T \sin(\bullet - \theta )\, \rho(\theta) \, \mathrm{d}\theta - K_s \sin(2 \bullet) .
    \end{equation}
\end{corollary}

\begin{proof}
    Since $\rho = \theta_\# \leb$, the density $\rho$ is a fixed point of \eqref{eulerian_dynamics} if and only if $\theta$ is a fixed point of \eqref{flow_eqn}, i.e.~if and only if $v_\theta(\theta(x)) = 0$ for (almost) all $x$. By the pushforward relation, $\theta$ is valued in a set $A \subset \T$ for (almost) all $x$ if and only if $\rho$ has support contained in $A$.
\end{proof}

In the above, $P(\T)$ denotes the space of probability measures on $\T$, i.e., the state space in the Eulerian representation. In this section, we will show that the velocity field \eqref{v_rho_eqn} must have either two, three, or four zeros, and thus at each fixed point, $\rho$ must consist of a discrete density (i.e.~a sum of Dirac masses) with at most four components. This will make it possible to enumerate all of the equilibrium configurations for the system.

Starting from the order parameter representation of the velocity field \eqref{order_param_v} and writing the first order parameter in polar form as $z_1 = |z_1|e^{i\angle z_1}$, we may equivalently express $v$ in terms of the magnitude and phase of $z_1$ as
\begin{equation} \label{v_z_phase_eqn}
    v_{z_1}(\theta) ~=~ - K|z_1|\sin(\angle z_1 - \theta) - K_s \sin(2\theta) .
\end{equation}
Thus $v$ depends on $\rho$ only through $z_1 = z_1(\rho)$, and is finitely parameterized by the four real variables $|z_1|$, $\angle z_1$, $K$, and $K_s$. Note also that $v$ is real analytic and cannot vanish identically, since we suppose that $K_s$ is strictly positive.

We may therefore classify all possible $v$ (thus their zero sets, thus all possible equilibrium configurations) in terms of these parameters. Observing that since the zero set of any velocity field $v$ is invariant under any (nonzero) rescaling of $v$, we may reduce the study of the zero set of \eqref{v_z_phase_eqn} to the study of the zero set of the scaled velocity field
\begin{equation}
    \tilde{v}_{z_1}(\theta) ~:=~ - \frac{K}{K_s}|z_1|\sin(\angle z_1 - \theta) - \sin(2\theta) ~=~ - A \sin(\angle z_1 - \theta) - \sin(2\theta) , \label{eqn: velocity field normalized}
\end{equation}
where $A := K |z_1| / K_s$. We therefore only need to search over the two-parameter space $(A,\angle z_1)$ to determine all possible zero sets. (Note, however, that since $|z_1| \leq 1$, not every zero set configuration is feasible for a given choice of $K,K_s$.) This form for $v$ then allows us to conclude the following.

\begin{lemma} \label{zero_set_lem}
    There exist exactly two, three, or four zeros of the velocity field \eqref{eqn: velocity field normalized}. 
\end{lemma}

\begin{proof}
    First, rewrite \eqref{eqn: velocity field normalized} as
    \begin{align}
        \tilde{v}_{z_1}(\theta) &= -A\sin(\angle z_1 - \theta) - \sin(2\theta)\\
        &= -A\sin(\angle z_1)\cos{(\theta)} + A\sin{(\theta)}\cos{(\angle z_1)} - 2\sin{(\theta)}\cos{(\theta)} .
        \label{eqn: simplified velocity field}
    \end{align}
    Define $a := -A\sin{(\angle z_1)}, ~b := A\cos{(\angle z_1)}$ to obtain
    \begin{equation}
        \tilde{v}_{z_1}(\theta) = a\cos{\theta} + b\sin{\theta} - 2\sin{\theta}\cos{\theta} .
    \end{equation}
    Additionally, define $x := \cos{(\theta)}$ and $y:=\sin(\theta)$, and observe that the zeros of \eqref{eqn: simplified velocity field} are then precisely the points of intersection of the two conics
    \begin{align}
        H_{1} &:= \{ ax + by - 2xy = 0 \} \\
        C_{1} &:= \{ x^{2}+y^{2}-1 = 0 \} ,
    \end{align}
    a rectangular hyperbola and circle, respectively. By Bézout's theorem -- noting that $H_1$ and $C_1$ share no common component, since $C_1$ is irreducible and is not contained in $H_1$ -- we can conclude that the two conics can intersect in at most four points. Additionally, we note that the hyperbola $H_{1}$ always passes through the origin, and therefore, for all $a,b$, the curves $C_{1}$ and $H_{1}$ intersect in at least two points. Therefore, there exist exactly two, three, or four zeros of the velocity field \eqref{eqn: simplified velocity field}.
\end{proof}

In addition to the equilibrium condition that $\rho$ be supported on the zero set of $v_{z_1}$ (written formally as $v_{z_1} \rho = 0$), there is an additional \emph{compatibility condition} that $z_1$ actually be the order parameter corresponding to the given $\rho$. The following lemma makes this correspondence precise.

\begin{lemma}\label{lem: density as sum of dirac masses}
    Given $z_1$, $K$, $K_s$, consider the velocity field $v = v_{z_1}$ as defined in \eqref{v_z_phase_eqn}. Then:
    \begin{enumerate}
        \item There exists a $\rho$ such that $v_{z_1} \rho = 0$ and $z_1 = z_1(\rho)$ if and only if $z_1$ lies in the convex hull of the zero set of $v_{z_1}$.
        \item Supposing (1) holds, then every such $\rho$ takes the form $\rho = \sum_i m_i \delta_{\zeta_i}$ for some convex coefficients $\{m_i\}$ of $z_1$: $z_1 = \sum_i m_i \zeta_i$, $\sum_i m_i = 1$, $m_i \geq 0$, and where $\{\zeta_i\}$ denotes the zero set of $v_{z_1}$.
    \end{enumerate}
\end{lemma}

\begin{proof}
    To show the forward direction in (1), suppose that there exists a $\rho$ satisfying $v_{z_1} \rho = 0$ and $z_1 = z_1(\rho)$. Then by the interpretation of the order parameter $z_1$ as the complex center of mass of $\rho$ (given in Section \ref{subsec: constant k case}), $z_1$ must be contained in the convex hull of the support of $\rho$, which in turn must be contained in the convex hull of the zero set of $v_{z_1}$ (since $\rho$ must be supported in this set).

    To show the reverse direction in (1), suppose that $z_1$ belongs to the convex hull of the zero set of $v_{z_1}$, which we denote $\{ \zeta_i \}$. Then there exist convex coefficients $m_i$ such that $z_1 = \sum_i m_i \zeta_i$, $m_i \geq 0$, $\sum_i m_i = 1$. It is immediate to verify that $\rho := \sum_i m_i \delta_{\zeta_i}$ satisfies both $v_{z_1} \rho = 0$ and $z_1 = z_1(\rho)$. To show (2), observe that all $\rho$ satisfying these conditions must be of this form for some $\{ m_i \}$, since (a), $\rho$ is a probability density supported on the zero set of $v_{z_1}$ if and only if $\rho = \sum_i m_i \delta_{\zeta_i}$ for some convex coefficients $\{ m_i\}$, and (b), $z_1$ is the center of mass of said $\rho$ if and only if $z_1 = \sum_i m_i \zeta_i$.
\end{proof}

A depiction of the conic sections that define the zero set of $v$ (as constructed in the proof of Lemma \ref{zero_set_lem}) and the corresponding convex hulls are shown in Figure \ref{fig: solution intersection points}.

\begin{figure}[!h]
    \centering
    \includegraphics[width=\textwidth]{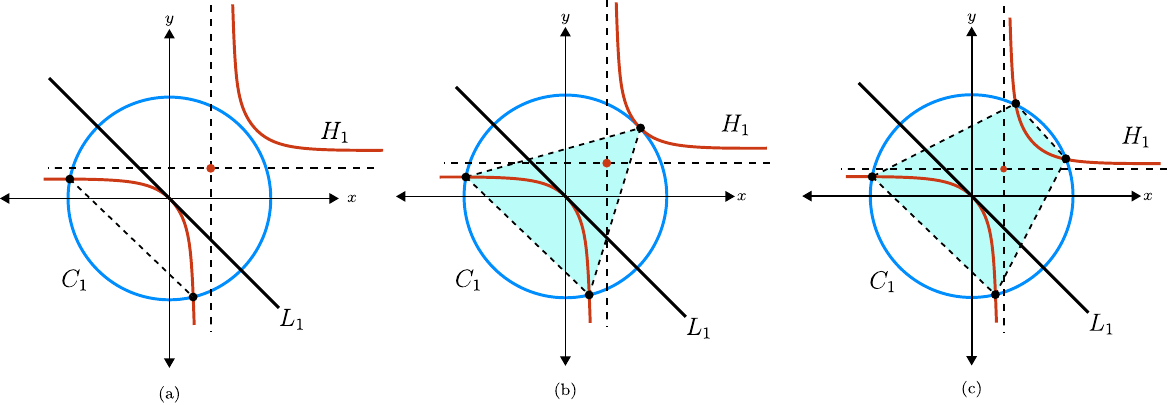}
    \caption{Diagrammatic description of the three cases of the intersection between the hyperbola $H_{1} = \{ ax+by- 2xy = 0 \}$ and the unit circle $C_{1} = \{ x^{2}+y^2- 1=0 \}$; (a) two points of intersection, (b) three points of intersection, (c) four points of intersection. The red point denotes the center of the hyperbola given by $(b/2,a/2)$, the black dots denote the points of intersection between $C_{1},H_{1}$. The black line $L_{1} = \{ by + ax = 0 \}$ is the tangent to $H_{1}$ at the origin and must contain the complex order parameter $z_1$ under the usual identification of $\mathbb{C}$ with $\R^2$. The shaded region in cyan denotes the convex hull formed by the points of intersection.}
    \label{fig: solution intersection points}
\end{figure}

Together, Lemmas \ref{zero_set_lem} and \ref{lem: density as sum of dirac masses} yield a constructive enumeration of the equilibria: sweeping over $(A,\angle z_1)$ determines all possible zero sets $\{ \zeta_i \}$ via the intersection of conics construction above; for each such zero set, the compatibility condition selects those values of $K/K_s$ (if any) for which $z_1$ falls within the convex hull of $\{ \zeta_i \}$; and the corresponding sets of convex coefficients $\{ m_i \}$ then give all equilibria $\rho = \sum_i m_i \delta_{\zeta_i}$. Note that while the \emph{zero sets} are a function of just two parameters, $K|z_1|/K_s$ and $\angle z_1$, the \emph{equilibrium points} are a function of three: $K/K_s$, $|z_1|$, and $\angle z_1$.

\section{Stability of Equilibria} \label{sec: stability}

So far, we have looked at the first-order (equilibrium) conditions for the Eulerian dynamics \eqref{eulerian_dynamics}. In this section, we shift our focus to the second-order conditions, which shed light on the stability of the equilibria identified in Section \ref{sec: equilibria}. Here, however, we run into the difficulties anticipated in Section \ref{subsec: wasserstein gradient flows}. Working on the interior of the Wasserstein space (option (a) there) is not available to us, since our equilibria are atomic; and working formally at these atomic points permits one to perturb an equilibrium only by transporting its masses, since Wasserstein perturbations are generated by vector fields. Such perturbations cannot break a discrete component apart, which is in conflict with the physical assumptions underlying the model. The metric-space theory (option (b)) is valid at these points, but supplies no second-order calculus there: it characterizes gradient flows through metric slopes and variational inequalities rather than through a Hessian, and so offers no object against which to test perturbations. Thus, we elect to work in the Lagrangian setting for this portion of the analysis.

Before we state the stability results formally, we briefly recall our discussion from Section \ref{subsec: fixed points and stability}. The conditions we compute are conditions for \emph{directional local minima}, which are necessary for asymptotic stability in the infinite-dimensional setting, and whose strong form is sufficient for asymptotic stability in the finite-dimensional approximations. Thus, the stability conditions that we provide fully describe the behavior of fixed points as observed in practice.

The next result follows immediately from a direct application of Lemma \ref{stability_lemma} to the particular form for $\nabla^2 E$.

\begin{lemma} \label{second_variation_lemma}
    For a fixed point $\theta$ to be a directional local minimum, it is necessary that the quadratic form
    \begin{multline} \label{quadratic_form}
        D^2E(\theta)[p,p] = \int \left( 2K_s \cos(2 \theta(x)) - K \int \cos(\theta(x) - \theta(y)) \, \mathrm{d}y \right) p^2(x) \, \mathrm{d}x \\
        + K \iint \cos(\theta(x) - \theta(y)) p(y) p(x) \, \mathrm{d}y \, \mathrm{d}x
    \end{multline}
    be positive semi-definite, that is, $\geq 0$ for all perturbations $p \in L^2$. Alternatively, it is sufficient for $\theta$ to be a \emph{strong} directional local minimum that the quadratic form \eqref{quadratic_form} be uniformly positive definite, that is, $\geq c \, \| p \|_{L^2}^2$ for some $c > 0$ and all $p \in L^2$.
\end{lemma}

\begin{proof}
    Apply Lemma \ref{stability_lemma} to the Hessian computed in Proposition \ref{regularity_prop_2}, specialized to the constant graphon $K(x,y) \equiv K$.
\end{proof}

The above condition is unwieldy to work with directly, since it is not possible to isolate $p$. However, while it is hard to obtain sufficient conditions from this form, it is relatively easy to obtain necessary ones: one simply tests the form against various perturbations $p$, and concludes that any equilibria which are destabilized by $p$ are certainly unstable in general. We therefore proceed by using a sequence of such perturbations to progressively narrow the candidate stable equilibria down to a manageable class, and then check sufficient conditions on that class alone. After a lengthy argument involving various symmetries, transformations, perturbations, and cases (all relegated to Appendix \ref{necessary_stability_appendix}), we arrive at the following.

\begin{theorem} \label{necessary_stability_thm}
    For almost all values\footnote{The single remarkable exception happens when $K/K_s = -2$ in the very particular configuration of two masses, each of mass $0.5$, at $|\theta_1| \leq \pi/4$ and $\theta_2 = \pi - \theta_1$. In this case, the equilibria appear to be neutrally stable.} of parameters $K/K_s$, for a fixed point $\theta$ to be a directional local minimum, it is necessary that it be valued in the set $\{ 0 , \pi \}$.
\end{theorem}

\begin{proof}
    See Appendix \ref{necessary_stability_appendix}.
\end{proof}

In the Eulerian setting, this translates to: for almost all values of parameters $K/K_s$, for a fixed point $\rho$ to be a directional local minimum, it is necessary that it be supported on $\{ 0 , \pi \}$.

Next, following the approach outlined above, we check sufficient conditions for the stability of phase functions valued in the set $\{ 0, \pi \}$. By the invariance of the dynamics under relabeling (see Sections \ref{sec: problem description} and \ref{subsec: constant k case}), we may assume without loss of generality that $\theta(x) = 0$ for $x \in [0,m)$ and $\theta(x)=\pi$ for $x \in (m,1]$. In general, the stability of these configurations depends both on the value of $m$ (i.e.~on the relative amount of mass at $0$ and $\pi$) and on the value of the parameter $K/K_s$. By finding tight bounds on the quadratic form \eqref{quadratic_form} in this case (relegated to Appendix \ref{sufficient_stability_appendix}), we arrive at the following.

\begin{theorem} \label{sufficient_stability_thm}
    Let $\theta_m$ be defined by
    \begin{equation} \label{binarized_family_eq}
        \theta_m(x) = \begin{cases}
            0 \quad \text{if} \quad x \in [0,m) \\
            \pi \quad \text{if} \quad x \in (m,1]
        \end{cases} .
    \end{equation}
    Then the following conditions are necessary for $\theta_m$ to be a directional local minimum:
    \begin{itemize}
        \item If $m \in \{ 0,1 \}$: $K/K_s \leq 2$.
        \item If $m = 0.5$: $-2 \leq K/K_s$.
        \item If $m \notin \{ 0, 0.5, 1 \}$: $-2 \leq K/K_s \leq 2 / |2m - 1|$.
    \end{itemize}
    The same bounds are sufficient for $\theta_m$ to be a directional local minimum when the inequalities above are strict.
\end{theorem}

\begin{proof}
    See Appendix \ref{sufficient_stability_appendix}.
\end{proof}

This completes the classification of stability of equilibria in this system. Additionally, the above calculation identifies the so-called ``threshold for binarization'' in this model, that is, the threshold for which there exist nontrivial (i.e.~$m \notin \{0,1\}$) stable binarized states. Indeed, Theorem \ref{sufficient_stability_thm} allows us to identify a few distinct parameter thresholds:
\begin{itemize}
    \item \emph{Some} nontrivial binarized state is stable $~\Leftrightarrow~$ $-2 < K/K_s$.
    \item \emph{Every} binarized state is stable $~\Leftrightarrow~$ $-2 < K/K_s < 2$.
    \item \emph{Only} nontrivial binarized states can be stable $~\Leftrightarrow~$ $2 < K/K_s$.
\end{itemize}
Classically, it is the first of these which is called the ``threshold for binarization''.

Lastly, recalling the discussion from Section \ref{subsec: fixed points and stability}, it is worth pausing to observe that these binarized fixed points form a continuous family of strong directional local minima which are not true local minima. Indeed, the map $m \mapsto \theta_m$ is continuous in $\Theta$, and by the above analysis, there exists an interval in the family consisting of stable fixed points for any $-2 < K/K_s$. However, evaluating the energy \eqref{lagrangian_energy} along the family gives
\begin{equation} \label{binarized_energy_eq}
    E(\theta_m) ~=~ \frac{K}{2} (2m-1)^2 - \frac{K_s}{2} ,
\end{equation}
which is nonconstant for all $K \neq 0$, so that every neighborhood of a generic $\theta_m$ contains fixed points of strictly lower energy. Consequently, within these parameter regimes, there exist strong directional local minima which are not true local minima. Furthermore, the existence of this family shows that $E$ cannot be Fr\'echet-$C^2$ or satisfy a \L{}ojasiewicz inequality in general.

\section{Numerical Results}
\label{sec: numerical results}

Here, we present numerical evidence that the mean-field model correctly predicts behavior regimes in a wide variety of homogeneous random graphs. We consider four categories of homogeneous random graph, listed in Figure~\ref{fig: stability_fig}, and illustrate the experimental procedure for case (c), with the other cases being similar. First, we generate a $N = 200$-node Erd\H{o}s-R\'enyi graph with edge probability $p_{i}\in [0.5,1]$, a random connection weight $A_i$ such that the effective coupling $K_{i} = p_{i}A_{i}$ satisfies $K_{i}/K_s \in [-5,10]$, and a random initial mass fraction $m_i(0)$. For this graph instance $\mathcal{G}_{i}$, we initialize the OIM at a configuration $\theta(0) = \theta^{\ast} + \varepsilon \nu$ where $\theta^*$ is a binarized state with $m_i(0)$ mass at $0$ and $1-m_i(0)$ mass at $\pi$, $0 < \varepsilon\ll1$, and $\nu\sim\mathbf{N}(0,\sigma^{\ast})$ is a random perturbation. The system is then allowed to evolve and settle to an equilibrium point, and the final mass fraction $m_i(\infty)$ is recorded. The value of $K_i/K_s$ for this graph instance is plotted against the final mass fraction $m_i(\infty)$. This experiment is repeated many times to obtain the scatter plots shown in Figure~\ref{fig: stability_fig}.

Now, as the result in Theorem~\ref{sufficient_stability_thm} is only valid in the limit $N\to\infty$, performing the above for a finite $N$ induces an approximation error. Following Remark \ref{finite_correspondence_rmk}, however, it is precisely the first- and second-order information which transfers between the finite- and infinite-dimensional settings, so we expect the second-order conditions of Theorem~\ref{sufficient_stability_thm} to predict the observed behavior at large but finite $N$. As can be seen, our theoretical prediction for the continuum model agrees closely with the finite-dimensional numerical results for all four cases. It is this agreement which licenses the claim made in Section \ref{sec: stability}, namely that the directional conditions computed there are the ones which describe the behavior of OIMs as observed in practice.

\begin{figure}[!h]
	\centering
	\includegraphics[width=1.0\linewidth]{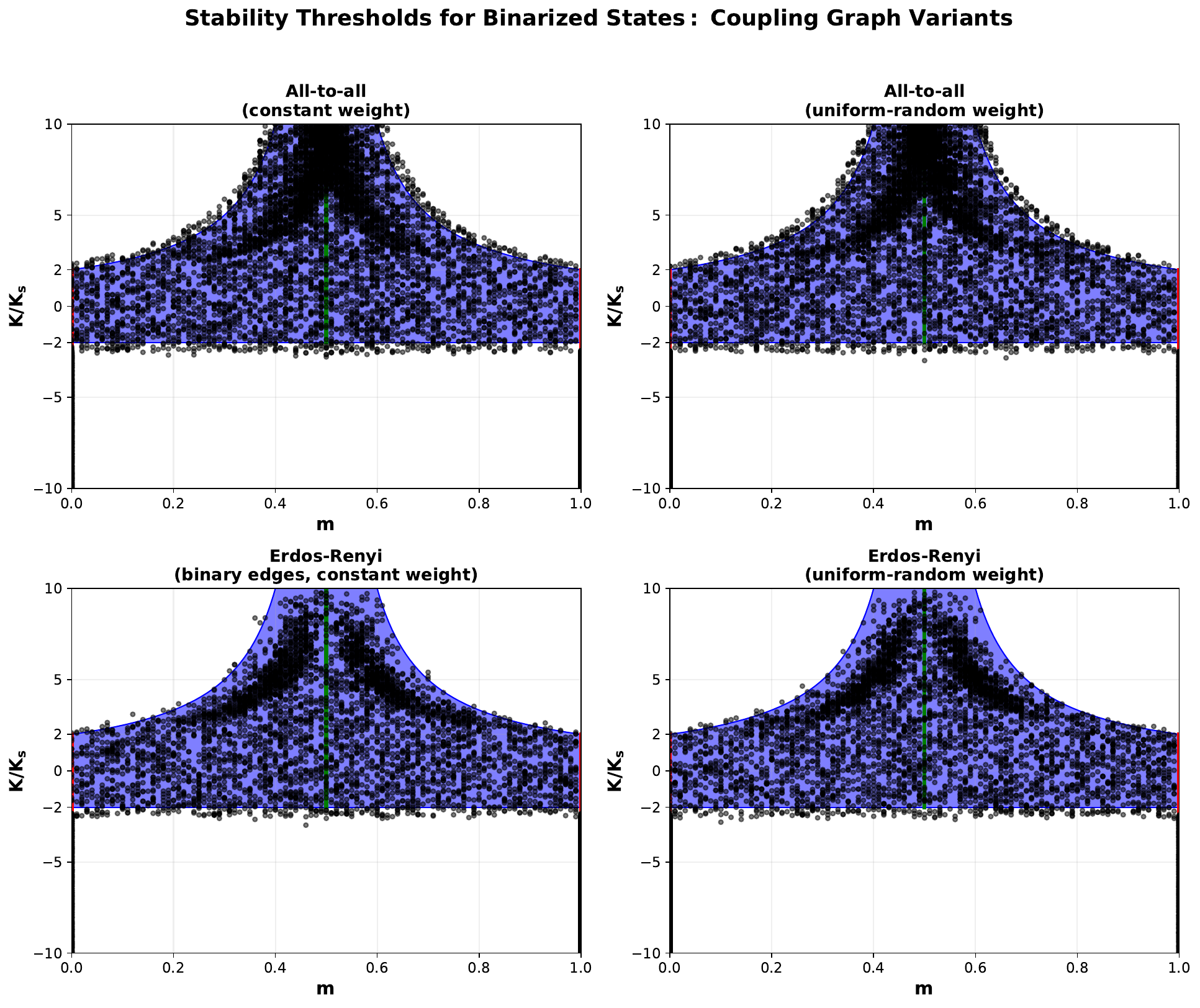}
	\caption{\textbf{Stability thresholds for binarized states.} \textbf{Theoretical prediction: } Binarized states are stable for $K/K_s<2$ when $m\in\{0,1\}$ (red), for $-2 < K/K_s$ when $m=0.5$ (green), and for $-2<K/K_s<2/|2m-1|$ otherwise (blue). \textbf{Experimental results: } Instances of binarized states in random networks are generated, with random connection probability, random connection strength, and random mass fraction. Each instance is perturbed and allowed to settle at a stable equilibrium, and the value of $K/K_s$ for that instance is recorded against the mass fraction $m$ of the final configuration; thus each data point represents a stable configuration for a given instance of a random network. The experiment is performed for the following cases: (a) all-to-all connectivity with constant weight, (b) all-to-all connectivity with uniform-random weight, (c) Erd\H{o}s-R\'enyi binary graph with constant weight, and (d) Erd\H{o}s-R\'enyi weighted graph with uniform-random weight. The blue region indicates the theoretical predictions in Theorem~\ref{sufficient_stability_thm}. Observe that below the threshold for binarization ($K/K_s < -2$), the only stable fixed points are single clusters at $0$ or $\pi$. As can be seen, there is a close agreement with the theoretical predictions.}
	\label{fig: stability_fig}
\end{figure}

\section{Concluding Remarks}
\label{sec: conclusion}

In this paper, we investigated mean-field limits of oscillator Ising machines. We derived two mean-field models: a Lagrangian one, which can always be defined, and an Eulerian one, which emerges as its quotient in the case of totally symmetric oscillators. We showed that both mean-field limits retain a gradient flow structure with respect to an appropriate geometry -- the Lagrangian with respect to the $L^2$ Riemannian geometry, and the Eulerian with respect to the 2-Wasserstein geometry, which arises in turn as the quotient of the former. We showed how this gradient-flow structure allows us to classify the limiting behavior of these models in terms of the Riemannian gradient and Hessian of the energy function, and discussed how the loss of regularity in the mean-field limit requires us to work with the weakened second-order condition of a directional local minimum. We also saw that where the mean-field analysis leaves questions open, the additional structure available in the finite-dimensional case can sometimes be brought to bear. We further discussed a degeneracy that occurs in the Eulerian representation at our fixed points, which forced us to work in the Lagrangian setting for some portions of the analysis. We then applied these conditions to the model at hand, showing that for almost all parameter values, only binarized fixed points can be stable, and we obtained tight bounds on the parameter thresholds for which this occurs. Lastly, we presented numerical results showing that the resulting classification accurately predicts behavior in large homogeneous random graphs.

Next, we compare and contrast our work with the existing literature and identify its merits and limitations. We see two main limitations of the present work, which we hope to address next. We recall, as mentioned in the introduction, that the authors in \cite{YC::2024::Chaos} broadly classify the equilibrium points of the finite-$N$ model. Here, we were able to provide an enumeration of the equilibrium points, study their stability, and provide estimates of binarization bounds. However, we have only treated the simplest version of the mean-field model here, namely the constant-graphon case. It is precisely this assumption of constant connectivity that has enabled us to make such strong claims about the model. In reality, the adjacency matrix $A$ is defined by the optimization problem at hand, and we expect it to be a richly structured object rather than something homogeneous. In fact, recent work on the mean-field analysis of Kuramoto oscillator networks under graphon limits has been successful \cite{HC::2019a::DCDS, HC::2019b::DCDS}, but the second-harmonic term of the OIM interferes with the standard machinery due to the lack of rotational invariance. Despite this, we hope to extend our analysis to more general graphons, for example, the graphon limits of power-law or small-world networks for OIMs. 

Second, the classification of limit solutions, while informative, is not the ultimate question which we hope to address. Rather, it is desirable to understand the statistical relationship between the quality of fixed points and the frequency with which they appear. In particular, this would then shed light on how the OIM parameters may be designed -- perhaps even scheduled or controlled -- in order to obtain the best performance from these devices. Relevant work in the literature in the finite $N$ case include \cite{MB::2023::JApplPhys} where the authors look at the relationships between the second harmonic injection strength $K_{s}$ and the stability of the optimal configurations. They note that the scheduling used for $K_{s}$ (see, e.g., \cite{TW::2021::NatComput}) could potentially trap the OIM in a sub-optimal local minima and thus affect the solution quality. We hope the mathematical tools and classification developed here thus lay a foundation from which to approach these deeper questions.

In turn, we expect both of these thrusts to push forward the tools which are available for the modeling, analysis, and control of large-scale systems.

\appendix

\section{Proof of Proposition \ref{regularity_prop}} \label{regularity_prop_proof}

We first compute the G\^ateaux derivative by Taylor expansion, using the definition \eqref{first_variation_eq}:
\begin{multline}
    DE(\theta)[v] = -\frac{1}{2} \int \int K(x,y) \sin(\theta(x)-\theta(y)) \, (v(x) - v(y)) \, \mathrm{d}x \, \mathrm{d}y \\
    + K_s \int \sin(2\theta(x)) \, v(x) \, \mathrm{d}x .
\end{multline}
Next we isolate $v(x)$ by using a change of variables, together with the symmetry of $K$, to turn the term $v(y)$ into $v(x)$
\begin{equation}
    DE(\theta)[v] = \int \left( - \int K(x,y) \sin(\theta(x) - \theta(y)) \, \mathrm{d}y + K_s \sin(2\theta(x)) \right) v(x) \, \mathrm{d}x
\end{equation}
and then apply \eqref{gradient_eq} to obtain the candidate gradient
\begin{equation}
    \nabla E(\theta)(x) = - \int K(x,y) \sin(\theta(x) - \theta(y)) \, \mathrm{d}y + K_s \sin(2\theta(x)) .
\end{equation}
It remains to check (a) that $DE(\theta)$ is a bounded operator, (b) that $DE$ satisfies the Fr\'echet differentiability condition
\begin{equation} \label{frechet_condition}
    \frac{| E(\theta + v) - E(\theta) - DE(\theta)[v] |}{\|v\|_{L^2}} \to 0 \quad \text{as} \quad \|v\|_{L^2} \to 0 ,
\end{equation}
and (c) that $\nabla E$ is Lipschitz. We proceed with one condition at a time. 
\begin{enumerate}
    \item[(a)] \textbf{$DE(\theta)$ is a bounded operator}: Let $|\I| := \int_\I 1 \, \mathrm{d}x$ denote the volume of $\I$. Supposing that $K$ is uniformly bounded by $\| K \|_\infty$, bounding $\sin$ by 1, and using Cauchy-Schwarz, we have
\begin{equation}
    | DE(\theta)[v] | \leq (\| K \|_\infty |\I| + K_s) \int |v(x)| \, \mathrm{d}x \leq (\| K \|_\infty |\I| + K_s) |\I|^{1/2} \| v \|_{L^2} .
\end{equation}
Thus, $DE$ is bounded, and by the Riesz representation theorem, $\nabla E$ is well-defined. Note also that $\nabla E$ is bounded on $\Theta$: from the formula above, $| \nabla E(\theta)(x) | \leq \| K \|_\infty |\I| + K_s$ pointwise, and therefore $\| \nabla E(\theta) \|_{L^2} \leq (\| K \|_\infty |\I| + K_s) |\I|^{1/2}$ uniformly in $\theta$.
 
\item[(b)] \textbf{$DE$ satisfies the Fr\'echet differentiability condition}: Since $|\cos(a+h) - \cos(a) + \sin(a) h| \leq \tfrac{1}{2} h^2$ pointwise for all $a,h \in \R$, the remainder
\begin{equation}
    R(\theta,v) := E(\theta + v) - E(\theta) - DE(\theta)[v]
\end{equation}
satisfies
\begin{equation}
    |R(\theta,v)| \leq \frac{\| K \|_\infty}{4} \int \int (v(x) - v(y))^2 \, \mathrm{d}x \, \mathrm{d}y + K_s \int v^2(x) \, \mathrm{d}x .
\end{equation}
Expanding the square in the first term and using $\big( \int v \big)^2 \geq 0$, we thus obtain
\begin{equation}
    |R(\theta,v)| \leq (\tfrac{1}{2} \| K \|_\infty |\I| + K_s) \| v \|_{L^2}^2 ,
\end{equation}
and the Fr\'echet differentiability follows.

\item[(c)] \textbf{$\nabla E$ is Lipschitz}: First, observe that we may work in local coordinates, since locally Lipschitz with constant $L$ at each point implies globally Lipschitz with constant $L$. Bounding the differences of sines by the differences of their arguments, we obtain the pointwise bound
\begin{align}
    |\nabla E(\theta_1)(x) - \nabla E(\theta_2)(x)| ~&\leq~ \| K \|_\infty \int \big| \theta_1(x) - \theta_2(x) \big| + \big| \theta_1(y) - \theta_2(y) \big| \, \mathrm{d}y \nonumber \\
    & \hspace{20mm} + 2K_s \big| \theta_1(x)-\theta_2(x) \big| \nonumber \\
    ~&=~ (\| K \|_\infty |\I| + 2 K_s) \big| \theta_1(x)-\theta_2(x) \big| + \| K \|_\infty \| \theta_1-\theta_2 \|_{L^1} .
\end{align}
Taking $L^2$ norms, using $\| \cdot \|_{L^1} \leq |\I|^{1/2} \| \cdot \|_{L^2}$, and collecting terms then gives
\begin{equation}
    \| \nabla E(\theta_1) - \nabla E(\theta_2) \|_{L^2} ~\leq~ 2 \big( \| K \|_\infty |\I| + K_s \big) \, \| \theta_1 - \theta_2 \|_{L^2} ,
\end{equation}
so that $\nabla E$ is Lipschitz with constant $L = 2(\| K \|_\infty |\I| + K_s)$.
\end{enumerate}

\section{Proof of Proposition \ref{regularity_prop_2}} \label{regularity_prop_2_appen}

We first compute the derivative of the gradient map. Fix $\theta \in \Theta$ and $p \in L^2$. Differentiating \eqref{lagrangian_gradient_eq} in the direction $p$, the difference quotients of the integrands are dominated by $\| K \|_\infty |p|$ and $2K_s |p|$ respectively, both of which lie in $L^1(\I)$, so we may differentiate under the integral sign to obtain
\begin{multline}
    \big( \nabla^2 E(\theta) \, p \big)(x) = \left( 2K_s \cos(2\theta(x)) - \int K(x,y) \cos\big(\theta(x)-\theta(y)\big) \, \mathrm{d}y \right) p(x) \\
    + \int K(x,y) \cos\big(\theta(x)-\theta(y)\big) \, p(y) \, \mathrm{d}y ,
\end{multline}
which is exactly the splitting \eqref{hessian_splitting_eq}.

For boundedness, note that $\| a(\theta) \|_{\infty} \leq 2 K_s + \| K \|_{\infty} |\I|$, and that the operator norm of a multiplication operator is the sup norm of its multiplier. The kernel of $T_\theta$ is bounded by $\| K \|_{\infty}$, so $\| T_\theta \, p \|_{L^2} \leq \| K \|_{\infty} |\I| \, \| p \|_{L^2}$ by Cauchy-Schwarz. Adding the two bounds gives $\| \nabla^2 E(\theta) \|_{\mathrm{op}} \leq 2K_s + 2 \| K \|_{\infty} |\I|$, uniformly in $\theta$.

For continuity in the strong operator topology, let $\theta_n \to \theta$ in $\Theta$ and fix $p \in L^2$. Note that we may we may assume that $\theta_n \to \theta$ pointwise almost everywhere by passing to a subsequence. Each of the three integrands appearing above then converges pointwise almost everywhere to its counterpart at $\theta$: this is immediate for $\cos(2\theta_n(x))$, and follows for the two integral terms by dominated convergence, with dominating functions $\| K \|_{\infty}$ and $\| K \|_{\infty} |p|$ respectively. Moreover, all three quantities are bounded uniformly in $n$ and $x$. A further application of dominated convergence in the variable $x$ -- with dominating functions $\big( 2 \| a \|_{\infty} \big)^2 p^2$ and $\big( \| K \|_{\infty} \| p \|_{L^1} \big)^2$ -- therefore gives
\begin{equation}
    M_{a(\theta_n)} \, p ~\longrightarrow~ M_{a(\theta)} \, p 
    \qquad \text{and} \qquad 
    T_{\theta_n} \, p ~\longrightarrow~ T_\theta \, p
\end{equation}
in $L^2$, which is the desired convergence.

\section{Proof of Theorem \ref{necessary_stability_thm}} \label{necessary_stability_appendix}

The simplest perturbation is the constant perturbation $p(x) \equiv p$. Applying this perturbation to the quadratic form \eqref{quadratic_form}, terms cancel, and we obtain
\begin{equation}
    D^{2}E(\theta)[p,p] = \left( \int 2K_s \cos(2 \theta(x)) \, dx \right) p^2 .
\end{equation}
For this quantity to be $\geq 0$ for all $p$, it is necessary that
\begin{equation} \label{constant_perturbation_condition}
    \int \cos(2 \theta(x)) \, dx ~\geq~ 0 .
\end{equation}
We interpret this condition as expressing that ``the values in $\theta$ (equivalently, the mass in $\rho$) must be concentrated sufficiently near the minima of the second harmonic injection $-\cos(2\theta)$, i.e., sufficiently near $\{0,\pi\}$.''

For our next perturbation, we take $p = \pm I(w,\varepsilon)$, a positive or negative indicator function on an $\varepsilon$ neighborhood of the point $w \in [0,1]$. This perturbation is a proxy for a perturbation of a single oscillator. Note that $\theta$ can be taken to be piecewise constant, since it must be valued in the zero set of $v_\theta$, which is discrete. Applying this perturbation to the quadratic form \eqref{quadratic_form}, we obtain
\begin{equation}
    \varepsilon \left( 2 K_s \cos(2 \theta(w)) - K \int \cos(\theta(w) - \theta(y)) \, dy \right) + K \varepsilon^2 .
\end{equation}
For small $\varepsilon$, clearly only the term in parentheses matters. We recognize this term as the negative derivative of the velocity field $v_\theta$ evaluated at $\theta(w)$. Thus, these perturbations are destabilizing if $\theta(w)$ is an unstable zero of $v_\theta$, that is, if $v_\theta'(\theta(w)) > 0$.

Thus, candidate stable equilibria $\theta$ must be valued in the set of stable zeros of $v$ almost everywhere (equivalently, $\rho$ must be supported on the set of stable zeros). Based on the analysis in Section \ref{sec: equilibria}, we know that $v$ can have exactly 2, 3, or 4 zeros. By standard topological arguments, the continuity of $v$ allows us to conclude that $v$ must have exactly 1 or 2 stable zeros. Thus, we may now restrict our attention to candidate stable equilibria $\theta$ having just one or two levels (equivalently, $\rho$ having just one or two discrete components). We consider these cases separately in the following two subsections.

\subsection{Single-Level Case}

In the case where $\theta$ has a single level (equivalently, $\rho$ has a single mass), we take $\theta(x) \equiv \theta$. We evaluate the dynamics \eqref{lagrangian_dynamics} to find that the interaction term vanishes, and are left with
\begin{equation}
    \dot{\theta} = - K_s \sin(2\theta) .
\end{equation}
We can see that these dynamics admit four possible equilibria at $0$, $\pi/2$, $\pi$, and $3\pi/2$. Applying the test \eqref{constant_perturbation_condition} allows us to rule out the equilibria at $\pi/2$ and $3\pi/2$, leaving only those at $0$ and $\pi$ as candidate stable equilibria.

\subsection{Double-Level Case}

We now consider the case where $\theta$ has two levels: $\theta(x) =\theta_1$ on $[0,m)$ and $\theta(x) = \theta_2$ on $(m,1]$. We assume that we are \emph{genuinely} in a two-mass case here (i.e.~that $m \notin \{0,1\}$, $\theta_1 \neq \theta_2$) since the single-mass case has already been treated above.

Both $\theta_1$ and $\theta_2$ must be valued in the set of stable zeros of $v$. Moreover, by the intersection of conics argument given in Section \ref{sec: equilibria}, the zeros of $v$ alternate between stable and unstable around $\T$, so that if there are two distinct stable zeros, these must lie in separate halves of the unit circle: one in $(-\pi/2,\pi/2)$ and one in $(\pi/2,3\pi/2)$. Without loss of generality, we may assume that $\theta_1 \in (-\pi/2,\pi/2)$ and $\theta_2 \in (\pi/2,3\pi/2)$.

Evaluating the resulting dynamics \eqref{lagrangian_dynamics} for each level, we obtain
\begin{align}
    \dot{\theta}_1 ~&=~ K(1-m) \sin(\theta_1-\theta_2) - K_s \sin(2\theta_1) \\
    \dot{\theta}_2 ~&=~ Km \sin(\theta_2 - \theta_1) - K_s \sin(2\theta_2) .
\end{align}
By our earlier scaling arguments in Section \ref{sec: equilibria}, we may equivalently consider the normalized dynamics
\begin{align}
    \dot{\tilde{\theta}}_1 ~&=~ A(1-m) \sin(\theta_1-\theta_2) - \sin(2\theta_1) \label{reduced_dynamics_1} \\
    \dot{\tilde{\theta}}_2 ~&=~ Am \sin(\theta_2 - \theta_1) - \sin(2\theta_2) , \label{reduced_dynamics_2}
\end{align}
where $A := K/K_s$.

The equilibrium condition here is that both of the terms above be zero simultaneously. This also means that any linear combination of these terms must be zero. For example, scaling the two equations by $m$ and $1-m$ and summing them, we obtain the condition
\begin{equation} \label{equilibrium_condition_1}
    m \sin(2\theta_1) + (1-m) \sin(2\theta_2) = 0 .
\end{equation}
Similarly, considering the difference of the two terms, we obtain
\begin{equation}
    A \sin(\theta_1 - \theta_2) - \sin(2 \theta_1) + \sin(2\theta_2) = 0 ,
\end{equation}
which can be simplified (using the difference of sines formula) to
\begin{equation}
    \sin(\theta_1 - \theta_2)(A - 2 \cos(\theta_1 + \theta_2)) = 0 .
\end{equation}
In particular, either $\sin(\theta_1 - \theta_2)=0$, or $A - 2 \cos(\theta_1 + \theta_2) = 0$. If $\sin(\theta_1 - \theta_2) = 0$, then since we are assuming $\theta_1 \neq \theta_2$, $\theta_1$ and $\theta_2$ must differ by $\pi$. Considering the dynamics \eqref{reduced_dynamics_1}-\eqref{reduced_dynamics_2} implies that $\sin(2\theta_1)=\sin(2\theta_2)= 0$, and thus either $\{\theta_1,\theta_2\}=\{0,\pi\}$ or $\{\theta_1,\theta_2\}=\{\pi/2,3\pi/2\}$. The latter case is ruled out by our earlier assumptions on $\theta_1,\theta_2$ (or by the test \eqref{constant_perturbation_condition}), leaving $\{\theta_1,\theta_2\}=\{0,\pi\}$ as the only candidate equilibria if $\sin(\theta_1-\theta_2)=0$. Thus, in what follows, we may assume that $A - 2 \cos(\theta_1 + \theta_2) = 0$.

We now test the quadratic form \eqref{quadratic_form} against the two-level perturbations $p(x) = p_1$ on $[0,m)$ and $p(x) = p_2$ on $(m,1]$. Writing $c := \cos(\theta_1 - \theta_2)$ and evaluating the integrals blockwise, the cross terms partially cancel and we obtain
\begin{equation} \label{restricted_form_eq}
    D^2E(\theta)[p,p] ~=~ K_s \, \begin{bmatrix} p_1 \\ p_2 \end{bmatrix}^\top H \begin{bmatrix} p_1 \\ p_2 \end{bmatrix} ,
\end{equation}
where
\begin{equation}
    H ~:=~ \begin{bmatrix}
        2m \cos(2\theta_1) - A m (1-m) c & A m (1-m) c \\
        A m (1-m) c & 2(1-m) \cos(2\theta_2) - A m (1-m) c
    \end{bmatrix} .
\end{equation}
By Lemma \ref{stability_lemma}, for $\theta$ to be a directional local minimum, it is necessary that $H$ be positive semidefinite, and in particular that $\det H \geq 0$. It is convenient to note that
\begin{equation} \label{H_MJ_eq}
    H ~=~ - M J , \qquad M := \operatorname{diag}(m, 1-m) ,
\end{equation}
where $J$ is the Jacobian matrix of the reduced dynamics \eqref{reduced_dynamics_1}-\eqref{reduced_dynamics_2}
\begin{equation} \label{Jacobian}
J ~:=~
    \begin{bmatrix}
        A (1-m) \cos(\theta_1 - \theta_2) - 2 \cos(2 \theta_1) & -A (1-m) \cos(\theta_1 - \theta_2) \\
        -A m \cos(\theta_1 - \theta_2) & A m \cos(\theta_1 - \theta_2) -2 \cos (2 \theta_2)
    \end{bmatrix} .
\end{equation}
Since $\det H = m(1-m) \det J$ and $m \in (0,1)$, the condition $\det H \geq 0$ is equivalent to $\det J \geq 0$, and it suffices to compute the latter.

To do so, we write $s := \theta_1 + \theta_2$ and $d := \theta_1 - \theta_2$, so that $2\theta_1 = s+d$ and $2\theta_2 = s-d$, and abbreviate $P := \cos(s)\cos(d)$ and $Q := \sin(s)\sin(d)$. Expanding, the $A^2$ terms cancel and we obtain
\begin{equation}
    \det J = -2A \cos(d) \big[ P + (1-2m) Q \big] + 4 \big( P^2 - Q^2 \big) .
\end{equation}
Our earlier condition $A - 2\cos(s) = 0$ gives $A \cos(d) = 2P$, and therefore
\begin{equation}
    \det J = -4 Q \big[ Q + (1-2m) P \big] .
\end{equation}
Now, the equilibrium condition \eqref{equilibrium_condition_1} reads $\sin(s)\cos(d) = (1-2m)\cos(s)\sin(d)$ in these variables. Multiplying $Q + (1-2m)P$ by $\sin(d)$ and substituting this identity gives $\big[ Q + (1-2m)P \big] \sin(d) = \sin(s)$, where we recall that $\sin(d) \neq 0$ since the case $\sin(\theta_1-\theta_2) = 0$ was treated above. We therefore arrive at the simple expression
\begin{equation} \label{det_J_eq}
    \det J = -4 \sin^2(\theta_1 + \theta_2) .
\end{equation}
Since \eqref{det_J_eq} is always nonpositive, the condition $\det J \geq 0$ forces $\sin(\theta_1+\theta_2) = 0$. Substituting this back into the equilibrium condition \eqref{equilibrium_condition_1} gives $(1-2m)\cos(s)\sin(d) = 0$, and since $\sin(d) \neq 0$ and $\cos(s) = \pm 1$, we conclude that $m = 0.5$.

Since $\theta_1 \in (-\pi/2,\pi/2)$ and $\theta_2 \in (\pi/2,3\pi/2)$, we must have $\theta_1 + \theta_2 = \pi$. Therefore $\cos(\theta_1+\theta_2) = -1$, and by our earlier condition $A - 2\cos(\theta_1+\theta_2) = 0$, we conclude that $A = -2$.

The test \eqref{constant_perturbation_condition} then allows us to further limit the possible ranges of $\theta_1$ and $\theta_2$ to $[-\pi/4,\pi/4]$ and $[3\pi/4,5\pi/4]$, respectively. For these equilibria ($m=0.5$, $A=-2$, $\theta_1, \theta_2$ in the given ranges, with $\theta_1+\theta_2 = \pi$), the quadratic form \eqref{restricted_form_eq} appears to be only positive semidefinite, that is, $H$ has a zero eigenvalue. And indeed, these equilibria appear to be neutrally stable.

\section{Proof of Theorem \ref{sufficient_stability_thm}} \label{sufficient_stability_appendix}

Throughout this appendix, we abbreviate the condition that $\theta$ be a strong directional local minimum as ``stability'', in keeping with the discussion of Section \ref{subsec: fixed points and stability}.

\textbf{Case 1:} We begin with the case where $m=1$, i.e., where $\theta(x) = 0$ on all of $[0,1]$. (The case with $m=0$ is symmetric and has the same stability threshold.) In this case, the quadratic form \eqref{quadratic_form} simplifies to
\begin{align}
    D^2E(\theta)[p,p] ~&=~ (2 K_s - K) \int p^2(x) \, dx + K \int p(x) p(y) \, dx \, dy \\
    ~&=~ (2K_s-K) \| p \|_{L^2}^2 + K \left( \textstyle \int p \right)^2 \\
    ~&=~ 2K_s \| p \|_{L^2}^2 - K \left[ \textstyle \int p^2 - \left(\textstyle \int p \right)^2 \right] .
\end{align}
We can upper and lower bound the bracketed term: $0 \leq \textstyle \int p^2 - \left(\textstyle \int p \right)^2  \leq \| p \|_{L^2}^2$. These upper and lower bounds are both achieved by perturbations $p$ which are constant and zero-mean, respectively. Thus we can obtain the following (tight) bounds on the quadratic form:
\begin{align}
    K ~ \text{positive} \quad &\Rightarrow \quad (2 K_s-K) \| p \|^2 ~\leq~ D^2E(\theta)[p,p] ~\leq~ 2 K_s \| p \|^2 \\
    K ~ \text{negative} \quad &\Rightarrow \quad 2 K_s \| p \|^2 ~\leq~ D^2E(\theta)[p,p] ~\leq~ (2 K_s - K) \| p \|^2 .
\end{align}
The strongest conclusion that can be drawn here is from the lower bound when $K$ is positive. That is, that $2K_s - K > 0$ or $K/K_s < 2$ is sufficient for stability. This same condition is necessary if the strict inequality is relaxed to a non-strict one, since the lower bound is achieved.

\textbf{Case 2:} Next, we treat the case where $m=0.5$. Using the identity $\cos(a-b) = \cos(a)\cos(b) + \sin(a)\sin(b)$, we obtain
\begin{equation}
    D^2E(\theta)[p,p] = 2 K_s \| p \|_{L^2}^2 + K \left( \textstyle\int_0^{0.5} p - \int_{0.5}^1 p \right)^2 = 2 K_s \| p \|_{L^2}^2 + K \left( \textstyle\int p \cos(\theta) \right)^2 .
\end{equation}
Here, we use the bounds $0 \leq \left( \textstyle\int p \cos(\theta) \right)^2 \leq \| p \|_{L^2}^2$,
which are achieved by $p$ having the same sign and opposite sign on $[0,m)$ and $(m,1]$, respectively. Thus we get the following (tight) upper and lower bounds for the quadratic form:
\begin{align}
    K ~ \text{positive} \quad &\Rightarrow \quad 2 K_s \| p \|^2 ~\leq~ D^2E(\theta)[p,p] ~\leq~ (2 K_s + K) \| p \|^2 \\
    K ~ \text{negative} \quad &\Rightarrow \quad (2 K_s + K) \| p \|^2 ~\leq~ D^2E(\theta)[p,p] ~\leq~ 2K_s \| p \|^2 .
\end{align}
The strongest conclusion that can be drawn here is from the lower bound when $K$ is negative. That is, $2K_s + K > 0$ or $K/K_s > -2$ is sufficient for stability. Again, this same condition is necessary if the strict inequality is relaxed to a non-strict one, since the lower bound is achieved.

\textbf{Case 3:} Now, we examine the case where $m \notin \{0,0.5,1\}$. We can use the same tricks, including the identity $\cos(a-b) = \cos(a)\cos(b) + \sin(a)\sin(b)$, to obtain
\begin{equation} \label{case_3_eqn}
    D^2E(\theta)[p,p] = 2K_s \| p \|_{L^2}^2 - K(2m-1) \textstyle\int \cos(\theta(x)) p^2(x) \, dx + K ( \int \cos(\theta(x)) p(x) \, dx)^2 .
\end{equation}
Suppose first that $K$ is positive. Bounding the middle term below by $-K|2m-1| \, \| p \|^2$ and the last term below by $0$, we obtain
\begin{equation}
    (2 K_s - K|2m-1|) \| p \|^2 ~\leq~ D^2E(\theta)[p,p] ,
\end{equation}
so that $2K_s - K|2m-1| > 0$, or equivalently $K/K_s < 2/|2m-1|$, is sufficient for stability. Moreover, this bound is tight: both terms attain their lower bounds simultaneously for $p$ supported on $(m,1]$ if $2m-1$ is negative, or on $[0,m)$ if $2m-1$ is positive, and having zero mean on that set.

Suppose now instead that $K$ is negative. Bounding the two terms separately as above is no longer tight, since their lower bounds are now competing, and so we require a more careful analysis. Normalizing $\| p \|_{L^2} = 1$, we may write the value of \eqref{case_3_eqn} as
\begin{equation}
    2 K_s - K C(m,p) , \qquad
    C(m,p) := (2m-1) \textstyle\int \cos(\theta(x)) p^2(x) \, dx - ( \int \cos(\theta(x)) p(x) \, dx)^2 .
\end{equation}
Since $K$ is negative, finding a lower bound on \eqref{case_3_eqn} is equivalent to finding a lower bound on $C(m,p)$. Breaking the integrals over the sets $[0,m]$ and $(m,1]$ and using the fact that $\cos(\theta(x))=1$ on $[0,m]$ and $-1$ on $(m,1]$, we define
\begin{equation}
    a := \int_0^m p^2 , ~~ b := \int_m^1 p^2 , ~~ u := \int_0^m p , ~~ v := \int_m^1 p ,
\end{equation}
so that
\begin{equation}
    C(m,p) = (2m-1)(a-b) - (u-v)^2
\end{equation}
under the constraints $a,b \geq 0$, $a+b=1$, and the Cauchy-Schwarz bounds $u^2 \leq ma$, $v^2 \leq (1-m)b$.

For any fixed $a,b$, the term $(u-v)^2$ is maximized by taking $u$ and $v$ at their Cauchy-Schwarz bounds and with opposite signs, which is achieved by taking $p$ constant on each of the two sets, with opposite signs: $p = \sqrt{a/m}$ on $[0,m]$ and $p = -\sqrt{b/(1-m)}$ on $(m,1]$. Substituting these values, together with $b = 1-a$, gives the tight lower bound
\begin{equation}
    (2m-1)(2a-1) - ma - (1-m)(1-a) - 2 \sqrt{m(1-m)a(1-a)} ~\leq~ C(m,p) . 
\end{equation}
It is straightforward to verify that the left-hand side is minimized by the choice $a = 1-m$, where it attains the value $-1$. We therefore have the tight lower bound $-1 \leq C(m,p)$, achieved for $p = \sqrt{(1-m)/m}$ on $[0,m]$ and $-\sqrt{m/(1-m)}$ on $(m,1]$, and hence
\begin{equation}
    2 K_s + K \leq D^2E(\theta)[p,p] .
\end{equation}
Thus $0 < 2K_s + K$, or equivalently $K/K_s > -2$, is sufficient for stability when $K$ is negative. Combining the two cases gives the sufficient stability bounds
\begin{equation}
    -2 < K/K_s < 2 / |2m-1| .
\end{equation}
By our earlier arguments, these bounds are achieved, and so they are also necessary for stability if the inequalities are relaxed to be non-strict.

\section*{Acknowledgments}

The authors would like to thank Ahmed Allibhoy, Matt Jacobs, Jeff Moehlis and Shiv Chandrasekaran for insightful discussions during the development of this work.

\bibliographystyle{unsrt}
\bibliography{references}

\begin{thebibliography}{10}

\bibitem{AL::2014::FrontPhys}
Andrew Lucas.
\newblock Ising formulations of many {NP} problems.
\newblock {\em Frontiers in Physics}, 2:5, 2014.

\bibitem{RK::2010::IntProg}
Richard~M. Karp.
\newblock Reducibility among combinatorial problems.
\newblock In {\em 50 Years of Integer Programming 1958-2008}, pages 219--241.
  Springer, 2010.
\newblock Year corrected to 2010 per Springer's own citation (some sources list
  2009, the book's copyright/press date).

\bibitem{TW::2019::UCNC}
Tianshi Wang and Jaijeet Roychowdhury.
\newblock {OIM}: Oscillator-based ising machines for solving combinatorial
  optimisation problems.
\newblock In {\em Unconventional Computation and Natural Computation (UCNC
  2019)}, volume 11493 of {\em Lecture Notes in Computer Science}, pages
  232--256. Springer, 2019.

\bibitem{TW::2021::NatComput}
Tianshi Wang, Leon Wu, Parth Nobel, and Jaijeet Roychowdhury.
\newblock Solving combinatorial optimisation problems using oscillator based
  {Ising} machines.
\newblock {\em Natural Computing}, 20(2):287--306, 2021.
\newblock Corrected from prior draft: middle author is Leon Wu, not Liwei Wu.

\bibitem{KT::2019::FPL}
Kosuke Tatsumura, Alexander~R. Dixon, and Hayato Goto.
\newblock {FPGA}-based simulated bifurcation machine.
\newblock In {\em 2019 29th International Conference on Field Programmable
  Logic and Applications (FPL)}, pages 59--66. IEEE, 2019.

\bibitem{SU::2011::OptExpress}
Shoko Utsunomiya, Kenta Takata, and Yoshihisa Yamamoto.
\newblock Mapping of {Ising} models onto injection-locked laser systems.
\newblock {\em Optics Express}, 19(19):18091--18108, 2011.
\newblock End page added (18108); prior draft listed only the start page.

\bibitem{JC::2019::SciRep}
Jonathan Chou, Suraj Bramhavar, Siddhartha Ghosh, and William Herzog.
\newblock Analog coupled oscillator based weighted {Ising} machine.
\newblock {\em Scientific Reports}, 9:14786, 2019.

\bibitem{SD::2021::NatElectron}
Sourav Dutta, Abhishek Khanna, Amey~S. Assoa, Hanjong Paik, Darrell~G. Schlom,
  Zoltan Toroczkai, Arijit Raychowdhury, and Suman Datta.
\newblock An {Ising} hamiltonian solver based on coupled stochastic
  phase-transition nano-oscillators.
\newblock {\em Nature Electronics}, 4:502--512, 2021.

\bibitem{WM::2022::NatElectron}
William Moy, Ibrahim Ahmed, Po-Wei Chiu, Kerem Camsari, Suman Datta, and Arijit
  Raychowdhury.
\newblock A 1,968-node coupled ring oscillator circuit for combinatorial
  optimization problem solving.
\newblock {\em Nature Electronics}, 5:310--317, 2022.

\bibitem{MB::2021::IEEEAccess}
Mohammad~Khairul Bashar, Antik Mallick, and Nikhil Shukla.
\newblock Experimental investigation of the dynamics of coupled oscillators as
  {Ising} machines.
\newblock {\em IEEE Access}, 9:148184--148190, 2021.

\bibitem{AM::2021::IEDM}
Antik Mallick, Mohammad~Khairul Bashar, Daniel~S. Truesdell, Benton~H. Calhoun,
  and Nikhil Shukla.
\newblock Overcoming the accuracy vs. performance trade-off in oscillator
  {Ising} machines.
\newblock In {\em 2021 IEEE International Electron Devices Meeting (IEDM)},
  pages 40.2.1--40.2.4. IEEE, 2021.

\bibitem{DN::2015::JXCDC}
Dmitri~E. Nikonov, Gyorgy Csaba, Wolfgang Porod, Tadashi Shibata, Danny Voils,
  Dan Hammerstrom, Ian~A. Young, and George~I. Bourianoff.
\newblock Coupled-oscillator associative memory array operation for pattern
  recognition.
\newblock {\em IEEE Journal on Exploratory Solid-State Computational Devices
  and Circuits}, 1:85--93, 2015.

\bibitem{AM::2020::NatCommun}
Antik Mallick, Mohammad~Khairul Bashar, Dylan~Shea Truesdell, Benton~H.
  Calhoun, Siddharth Joshi, and Nikhil Shukla.
\newblock Using synchronized oscillators to compute the maximum independent
  set.
\newblock {\em Nature Communications}, 11:4689, 2020.

\bibitem{MB::2023::SciRep}
Mohammad~Khairul Bashar and Nikhil Shukla.
\newblock Designing {Ising} machines with higher order spin interactions and
  their application in solving combinatorial optimization.
\newblock {\em Scientific Reports}, 13:9558, 2023.
\newblock Corrected from prior draft: published 2023 in vol. 13 (not 2021, vol.
  11); DOI corrected accordingly.

\bibitem{YC::2024::Chaos}
Yi~Cheng, Mohammad~Khairul Bashar, Nikhil Shukla, and Zongli Lin.
\newblock A control theoretic analysis of oscillator {Ising} machines.
\newblock {\em Chaos: An Interdisciplinary Journal of Nonlinear Science},
  34(7):073103, 2024.

\bibitem{MB::2023::JApplPhys}
Mohammad~Khairul Bashar, Zongli Lin, and Nikhil Shukla.
\newblock Stability of oscillator {Ising} machines: Not all solutions are
  created equal.
\newblock {\em Journal of Applied Physics}, 134(14):144901, 2023.

\bibitem{YC::2026::PhysRevE}
Yi~Cheng and Zongli Lin.
\newblock Training oscillator {Ising} machines to assign the dynamic stability
  of their equilibrium points.
\newblock {\em Physical Review E}, 113(3):035304, 2026.

\bibitem{AA::2025::CDC}
Ahmed Allibhoy, Arthur~N. Montanari, Fabio Pasqualetti, and Adilson~E. Motter.
\newblock Global optimization through heterogeneous oscillator {Ising}
  machines.
\newblock In {\em 2025 IEEE 64th Conference on Decision and Control (CDC)},
  pages 998--1005, 2025.

\bibitem{MF::2025::arXiv}
Malihe Farasat, E.~M. H. E.~B. Ekanayake, and Nikhil Shukla.
\newblock Spin freezing in oscillator {Ising} machines: When second harmonic
  injection impedes computation.
\newblock {\em arXiv preprint}, 2025.

\bibitem{IA::2021::JSSC}
Ibrahim Ahmed, Po-Wei Chiu, William Moy, and Chris~H. Kim.
\newblock A probabilistic compute fabric based on coupled ring oscillators for
  solving combinatorial optimization problems.
\newblock {\em IEEE Journal of Solid-State Circuits}, 56(9):2870--2880, 2021.
\newblock Corrected from prior draft: fourth author is Chris H. Kim, not Arijit
  Raychowdhury.

\bibitem{SS::2000::PhysicaD}
Steven~H. Strogatz.
\newblock From {Kuramoto} to {Crawford}: Exploring the onset of synchronization
  in populations of coupled oscillators.
\newblock {\em Physica D: Nonlinear Phenomena}, 143(1-4):1--20, 2000.

\bibitem{YK::1975::LNP}
Yoshiki Kuramoto.
\newblock Self-entrainment of a population of coupled non-linear oscillators.
\newblock In Huzihiro Araki, editor, {\em International Symposium on
  Mathematical Problems in Theoretical Physics}, volume~39 of {\em Lecture
  Notes in Physics}, pages 420--422. Springer, 1975.

\bibitem{JA::2005::RevModPhys}
Juan~A. Acebr{\'o}n, Luis~L. Bonilla, Conrad~J. P{\'e}rez~Vicente, F{\'e}lix
  Ritort, and Renato Spigler.
\newblock The {Kuramoto} model: A simple paradigm for synchronization
  phenomena.
\newblock {\em Reviews of Modern Physics}, 77(1):137--185, 2005.

\bibitem{YK::2006::PTPSuppl}
Yoshiki Kuramoto, Shin-ichiro Shima, Dorjsuren Battogtokh, and Yuri Shiogai.
\newblock Mean-field theory revives in self-oscillatory fields with non-local
  coupling.
\newblock {\em Progress of Theoretical Physics Supplement}, 161:127--143, 2006.

\bibitem{HC::2013::ErgodicTheory}
Hayato Chiba.
\newblock A proof of the {Kuramoto} conjecture for a bifurcation structure of
  the infinite-dimensional {Kuramoto} model.
\newblock {\em Ergodic Theory and Dynamical Systems}, 35(3):762--834, 2013.
\newblock Published online Oct. 2013 (matches DOI/CiNii record); the final
  print issue is dated 2015 and is cited as such in some later papers.

\bibitem{HC::2019a::DCDS}
Hayato Chiba and Georgi~S. Medvedev.
\newblock The mean field analysis of the {Kuramoto} model on graphs {I}. the
  mean field equation and transition point formulas.
\newblock {\em Discrete and Continuous Dynamical Systems - Series A},
  39(1):131--155, 2019.

\bibitem{HC::2019b::DCDS}
Hayato Chiba and Georgi~S. Medvedev.
\newblock The mean field analysis of the {Kuramoto} model on graphs {II}.
  asymptotic stability of the incoherent state, center manifold reduction, and
  bifurcations.
\newblock {\em Discrete and Continuous Dynamical Systems - Series A},
  39(7):3897--3921, 2019.
\newblock DOI not confirmed in this session -- please verify (likely
  10.3934/dcds.2019xxx); year corrected/confirmed as 2019, not 2020 as in the
  "Chiba2020" source key.

\bibitem{RJ::1998::SIAMMathAnal}
Richard Jordan, David Kinderlehrer, and Felix Otto.
\newblock The variational formulation of the {Fokker-Planck} equation.
\newblock {\em SIAM Journal on Mathematical Analysis}, 29(1):1--17, 1998.

\bibitem{FO::2001::CommPDE}
Felix Otto.
\newblock The geometry of dissipative evolution equations: The porous medium
  equation.
\newblock {\em Communications in Partial Differential Equations},
  26(1-2):101--174, 2001.
\newblock Journal/volume/pages added -- omitted in the source .bib entry.

\bibitem{LA::2005::GradFlows}
Luigi Ambrosio, Nicola Gigli, and Giuseppe Savar{\'e}.
\newblock {\em Gradient Flows: in Metric Spaces and in the Space of Probability
  Measures}.
\newblock Springer, 2005.

\bibitem{JC::2003::RevMatIberoam}
Jos{\'e}~A. Carrillo, Robert~J. McCann, and C{\'e}dric Villani.
\newblock Kinetic equilibration rates for granular media and related equations:
  Entropy dissipation and mass transportation estimates.
\newblock {\em Revista Matem{\'a}tica Iberoamericana}, 19(3):971--1018, 2003.

\bibitem{BM::2010::M3AS}
Bertrand Maury, Aude Roudneff-Chupin, and Filippo Santambrogio.
\newblock A macroscopic crowd motion model of gradient flow type.
\newblock {\em Mathematical Models and Methods in Applied Sciences},
  20(10):1787--1821, 2010.

\bibitem{AK::1975::Book}
Andrey~Nikolaevich Kolmogorov and Sergei~Vasilyevich Fomin.
\newblock {\em Introductory Real Analysis}.
\newblock Dover / Courier Corporation, 1975.

\bibitem{SW::2003::Springer}
Stephen Wiggins.
\newblock {\em Introduction to Applied Nonlinear Dynamical Systems and Chaos}.
\newblock Springer, 2003.

\bibitem{YB::1991::CPAM}
Yann Brenier.
\newblock Polar factorization and monotone rearrangement of vector-valued
  functions.
\newblock {\em Communications on Pure and Applied Mathematics}, 44(4):375--417,
  1991.

\end{thebibliography}
\end{document}